\documentclass[11pt]{amsart}
\usepackage{amsmath,amsthm,amsfonts,amscd,amssymb,eucal,latexsym,mathrsfs,stmaryrd,enumitem,mathtools}
\usepackage[all]{xy}

\newtheorem{theorem}{Theorem}[section]

\newtheorem{lemma}[theorem]{Lemma}
\newtheorem{proposition}[theorem]{Proposition}

\theoremstyle{definition}
\newtheorem{definition}[theorem]{Definition}
\newtheorem{remark}[theorem]{Remark}

\newcounter{theoremintro}
\newtheorem{theoremi}[theoremintro]{Theorem}

\newcommand{\id}{{\rm id}}

\newcommand{\cC}{{\mathcal C}}

\newcommand{\cF}{{\mathcal F}}
\newcommand{\cV}{{\mathcal V}}

\newcommand{\fC}{{\mathfrak Y}}
\newcommand{\sA}{{\mathscr A}}
\newcommand{\sB}{{\mathscr B}}
\newcommand{\sC}{{\mathscr C}}
\newcommand{\sD}{{\mathscr D}}
\newcommand{\sE}{{\mathscr E}}

\newcommand{\sP}{{\mathscr P}}

\newcommand{\sS}{{\mathscr S}}

\newcommand{\fA}{{\mathfrak A}}

\newcommand{\Cb}{{\mathbb C}}

\newcommand{\Zb}{{\mathbb Z}}
\newcommand{\Pb}{{\mathbb P}}

\newcommand{\Nb}{{\mathbb N}}

\newcommand{\alg}{{\rm alg}}

\newcommand{\m}{{\rm m}}

\newcommand{\eps}{\varepsilon}

\newcommand{\Hamm}{{\rm Hamm}}

\newcommand{\hmax}{h}
\newcommand{\hinf}{{\underline{h}}}

\DeclareMathOperator{\Hom}{Hom}
\DeclareMathOperator{\Map}{Map}
\DeclareMathOperator{\Sym}{Sym}

\DeclareMathOperator{\supp}{supp}

\allowdisplaybreaks

\begin{document}

\title{Entropy, products, and bounded orbit equivalence}
\author{David Kerr}
\address{David Kerr,
Mathematisches Institut,
WWU M{\"u}nster,
Einsteinstr.\ 62,
48149 M{\"u}nster, Germany}
\email{kerrd@uni-muenster.de}

\author{Hanfeng Li}
\address{Hanfeng Li,
Center of Mathematics, Chongqing University, Chongqing 401331, China
and
Department of Mathematics, SUNY at Buffalo, Buffalo, NY 14260-2900, USA
}
\email{hfli@math.buffalo.edu}

\date{October 20, 2021}

\begin{abstract}
We prove that if two topologically free and entropy regular actions of countable sofic groups on compact metrizable spaces
are continuously orbit equivalent, and each group either (i) contains a w-normal amenable subgroup which is neither
locally finite nor virtually cyclic, or (ii) is a non-locally-finite product of two infinite groups,
then the actions have the same sofic topological entropy.
This fact is then used to show that
if two free uniquely ergodic and entropy regular probability-measure-preserving actions of such groups
are boundedly orbit equivalent then the actions have the same sofic measure entropy.
Our arguments are based on a
relativization of property SC to sofic approximations and yield more general entropy inequalities.
\end{abstract}

\maketitle

\tableofcontents

\section{Introduction}

At first glance it may seem that dynamical entropy and orbit equivalence should have little
to do with one another. One is a conjugacy invariant that is tailor-made
for the hairsplitting job of distinguishing Bernoulli shifts, all of which have the same spectral theory,
while the other is a coarse relation between group actions
whose tendency to nullify asymptotic behaviour is most devastating in the setting of amenable groups,
where entropy finds its classical home
\cite{Kol58,Kol59,Sin59,Dye59,OrnWei80}.
One registers information while the other threatens to destroy it.

This brutal disparity can however be honed so as to bring the two concepts
into frequent and sometimes surprising alignment.
Indeed entropy turns out to be sensitive in meaningful ways to the various kinds of restrictions that
one may naturally impose on an orbit equivalence, its role as an invariant remaining intact
in some cases but completely neutralized in others.
The history of this relationship
traces back several decades and in its original thrust encompasses
the work of Vershik on actions of locally finite groups \cite{Ver73,Ver95},
the Ornstein isomorphism machinery for Bernoulli shifts \cite{Orn70},
the theory of Kakutani equivalence \cite{Kat77,Fel76,OrnRudWei82,delRud84},
and Kammeyer and Rudolph's general theory of restricted orbit equivalence
for p.m.p.\ actions of countable amenable groups that all of this inspired
\cite{Rud85,KamRud97,KamRud02}
(see Chapter~1 of \cite{KamRud02} for a genealogy).
In a somewhat different vein from these lines of investigation,
Rudolph and Weiss later proved in \cite{RudWei00} that,
for a free p.m.p.\ action of a countable amenable group,
the conditional entropy with respect to a prescribed invariant sub-$\sigma$-algebra $\sS$ is preserved under
every $\sS$-measurable orbit equivalence.
As Rudolph and Weiss demonstrated in the application to
completely positive entropy that motivated their paper,
this crisp expression of complementarity
between entropy and orbit equivalence, when combined with the Ornstein--Weiss theorem \cite{OrnWei80},
turns out to be very useful
as a tool for lifting results from $\Zb$-actions to actions of general countable amenable groups.
More recently Austin has shown, for free p.m.p.\ actions of finitely generated amenable groups,
that entropy is an invariant of bounded and integrable orbit equivalence,
and that there is an entropy scaling formula for stable versions of these equivalences \cite{Aus16}.
It is interesting to note that Austin makes use of both the theory of Kakutani equivalence
(to handle the virtually cyclic case, which his approach requires him to treat separately)
and the Rudolph--Weiss theorem (in a reduction-to-$\Zb$ argument
which, ironically, forms part of the verification of the non-virtually-cyclic case).

The basic geometric idea at play in Austin's work when the group is not virtually cyclic
is the possibility of finding,
within suitable connected F{\o}lner subsets of the group, a connected subgraph which is sparse
but at the same time dense at a specified coarse scale.
By recasting this sparse connectivity as a condition on the action that we called {\it property SC}
and circumventing the ``derandomization'' of \cite{Aus16} with its reliance on the Rudolph--Weiss technique, we established in \cite{KerLi19} the following extension
beyond the amenable setting:
if $G$ is a countable group containing a w-normal amenable subgroup which is
neither locally finite nor virtually cyclic, $H$ is a countable group,
and $G\curvearrowright (X,\mu )$ and $H\curvearrowright (Y,\nu )$ are free
p.m.p.\ actions which are Shannon orbit equivalent (i.e., the cocycle partitions all have finite
Shannon entropy), then the maximum sofic measure entropies of the actions satisfy
\begin{gather}\label{E-SC}
h_\nu (H\curvearrowright Y) \geq h_\mu (G\curvearrowright X).
\end{gather}
One property shared by the groups $G$ in this theorem is that
their first $\ell^2$-Betti number vanishes, which in the nonamenable world can be
roughly intuited as an expression of anti-freeness, and indeed our approach
breaks down for free groups (see Section~3.5 of \cite{KerLi19}).
In what is surely not a coincidence,
groups whose Bernoulli actions are cocycle superrigid
also have vanishing first $\ell^2$-Betti number \cite{PetSin12},
and it has been speculated that these two properties
are equivalent in the nonamenable realm (curiously, however,
Bernoulli cocycle or orbit equivalence superrigidity remains generally
unknown for wreath products of the form $\Zb\wr H$ with $H$ nonamenable,
which satisfy the hypotheses on $G$ above).

Given that nonamenable products of countably infinite groups form a standard class
of examples within the circle of ideas around superrigidity, cost one, and vanishing first $\ell^2$-Betti number,
and in particular are known to satisfy
Bernoulli cocycle superrigidity by a theorem of Popa \cite{Pop08},
it is natural to wonder whether the entropy inequality (\ref{E-SC}) holds if $G$ is
instead assumed to be such a product.
In \cite{KerLi19} we demonstrated, in analogy with Gaboriau's result on cost for
products of equivalence relations \cite{Gab00}, that product actions of non-locally-finite product groups,
when equipped with an arbitrary invariant probability measure, satisfy property SC,
which is sufficient for establishing (\ref{E-SC}).
However, such actions always have maximum sofic entropy zero or $-\infty$.
One of the main questions motivating the present paper is whether one can
remove this product structure hypothesis on the action.

To this end we establish Theorem~\ref{T-measure 1} below, which gives the conclusion for
bounded orbit equivalence (i.e., orbit equivalence with finite cocycle partitions, as explained in Section~\ref{SS-bounded})
and uniquely ergodic actions.
We say that a p.m.p.\ action $G\curvearrowright (X,\mu )$ is {\it uniquely ergodic}
if the only $G$-invariant mean on $L^\infty (X,\mu )$ is integration with respect to $\mu$,
i.e., the induced action of $G$ on the spectrum of $L^\infty (X,\mu )$ is uniquely ergodic
in the usual sense of topological dynamics.
When $\mu$ is atomless, unique ergodicity forces the acting group to be nonamenable \cite[Theorem~2.4]{Sch81}.
In fact an ergodic p.m.p.\ action $G\curvearrowright (X,\mu )$ is uniquely ergodic
if and only if the restriction of the Koopman representation to $L^2 (X,\mu )\ominus \Cb 1$
does not weakly contain the trivial representation \cite[Proposition~2.3]{Sch81}. It follows that if $G$ is nonamenable
then unique ergodicity holds whenever the restriction of the Koopman representation to
the orthogonal complement of the constants is a
direct sum of copies of the left regular representation, and in particular
when the action has completely positive entropy \cite[Corollary~1.2]{Hay18}\cite[Corollary~1.7]{Sew19},
and thus occurs in the following examples:
\begin{enumerate}
\item Bernoulli actions $G\curvearrowright (X^G ,\mu^G )$,
where $(X,\mu )$ is a standard probability space and $(gx)_h = x_{g^{-1}h}$ for
all $g,h\in G$ and $x\in X^G$ (see Section~2.3.1 of \cite{KerLi16}),

\item algebraic actions of the form $G\curvearrowright (\widehat{(\Zb G)^n / (\Zb G)^n A} ,\mu )$
where $A\in M_n (\Zb G)$ is invertible as an operator on $\ell^2 (G)^{\oplus n}$ and $\mu$ is the
normalized Haar measure \cite[Corollary~1.5]{Hay19}.
\end{enumerate}
Moreover, if $G$ has property (T) then all of its ergodic p.m.p.\ actions are uniquely ergodic
\cite[Theorem~2.5]{Sch81}.

As above $\hmax_\mu (\cdot )$ denotes the maximum sofic measure entropy, and we
write $\hinf_\mu (\cdot )$ for the infimum sofic measure entropy
(see Section~\ref{SS-sofic measure entropy}).

\begin{theoremi}\label{T-measure 1}
Let $G$ and $\Gamma$ be countably infinite sofic groups at least one of which is not locally finite,
and let $H$ be a countable group.
Let $G\times\Gamma \curvearrowright (X,\mu )$ and $H\curvearrowright (Y,\nu )$ be
free p.m.p.\ actions which are boundedly orbit equivalent. Suppose that the action of $H$ is
uniquely ergodic. Then
\[
\hmax_\nu (H\curvearrowright Y) \ge \hinf_\mu (G\times\Gamma\curvearrowright X).
\]
\end{theoremi}

Theorem~\ref{T-measure 1} is a direct consequence of Theorem~\ref{T-bounded OE}
and Proposition~\ref{P-product}. When combined with Theorem~A of \cite{KerLi19} it yields the
following Theorem~\ref{T-measure 2}.
We say that an action is {\it entropy regular} if its maximum and infimum sofic entropies are equal,
i.e., the sofic entropy does not depend on the choice of sofic approximation sequence.
Entropy regularity for a p.m.p.\ action
is known to hold in the following situations:
\begin{enumerate}
\item the group is amenable, in which case
the sofic measure entropy is equal to the amenable measure entropy \cite{KerLi13,Bow12a},

\item the action is Bernoulli \cite{Bow10,KerLi11},

\item the action is an algebraic action of the form $G\curvearrowright (\widehat{(\Zb G)^n / (\Zb G)^n A} ,\mu )$
where $A\in M_n (\Zb G)$ is injective as an operator on $\ell^2 (G)^{\oplus n}$ and $\mu$ is the
normalized Haar measure \cite{Hay16},

\item the action is a shift action $G\curvearrowright \{ 1,\dots ,n\}^G$ 
equipped with a Gibbs measure
satisfying one of various uniqueness conditions \cite{Alp15,Alp17}.
\end{enumerate}
For the definition of w-normality see the paragraph before Theorem~\ref{T-w-normal amenable}.

\begin{theoremi}\label{T-measure 2}
Let $G$ and $H$ be countable sofic groups each of which either
\begin{enumerate}
\item contains a w-normal amenable subgroup which is neither locally finite nor virtually cyclic, or

\item is a product of two countably infinite sofic groups at least one of which is not locally finite.
\end{enumerate}
Let $G\curvearrowright (X,\mu )$ and $H\curvearrowright (Y,\nu )$ be free p.m.p.\ actions
which are uniquely ergodic and entropy regular, and suppose that they are boundedly orbit equivalent. Then
\[
h_\nu (H\curvearrowright Y) = h_\mu (G\curvearrowright X).
\]
\end{theoremi}

We note that, by a theorem of Belinskaya \cite{Bel68}, if two ergodic p.m.p.\ $\Zb$-actions are integrably
orbit equivalent, and in particular if they are boundedly orbit equivalent,
then they are measure conjugate up to an automorphism of $\Zb$
(what is referred to as ``flip conjugacy''). On the other hand,
a bounded orbit equivalence between ergodic p.m.p.\ $\Zb^d$-actions for $d\geq 2$
can scramble local asymptotic data to the point of scuttling
properties like mixing and completely positive entropy,
as Fieldsteel and Friedman demonstrated in \cite{FieFri86}.

Our strategy for proving Theorem~\ref{T-measure 1} is to localize
property SC to sofic approximations, yielding what we call ``property sofic SC''
for a group or an action, or more generally ``property $\sS$-SC'' where $\sS$ is a collection of sofic
approximations for the group in question (see Section~\ref{SS-sofic SC}).
The advantage of this localization is that the action itself need not have
a product structure, only the sofic approximation used to model it.
This accounts for the appearance of the infimum sofic entropy
in Theorem~\ref{T-measure 1}, in contrast to (\ref{E-SC}), but as noted
above many actions of interest are known to be entropy regular, in which
case one does in fact get (\ref{E-SC}). The trade-off in using property sofic SC
is its natural and frustratingly stubborn compatibility with the
point-map formulation of
sofic entropy, which is a kind of dualization of the homomorphism picture adopted in \cite{KerLi19}
and requires the choice of a topological model.
This has put us into the situation of not being able to control
the empirical distribution of microstates
except under the hypothesis of unique ergodicity, when the variational principle
makes such control unnecessary for the purpose of computing the entropy,
and even then we have had to restrict the hypothesis on the orbit equivalence from
Shannon to bounded.

Given that we are adhering to the point-map picture
with its use of topological models,
it makes sense to isolate as much of the argument as possible to the purely
topological framework, which also has its own independent interest.
Accordingly we establish the following theorem, the second part of which goes into
proving Theorem~\ref{T-measure 1} via Theorem~\ref{T-bounded OE}.
It is a direct consequence of Theorems~\ref{T-continuous OE} and \ref{T-w-normal amenable} and Proposition~\ref{P-product}.
Here $\hmax (\cdot )$ denotes the maximum sofic topological entropy
and $\hinf (\cdot )$ the infimum sofic topological entropy
(see Section~\ref{SS-sofic topological entropy}).

\begin{theoremi}\label{T-cts 1}
Let $G\curvearrowright X$ and $H\curvearrowright Y$ be
topologically free continuous actions of countable sofic groups on compact metrizable spaces, and
suppose that they are continuously orbit equivalent.
If $G$ contains a w-normal amenable subgroup which is neither locally finite nor virtually cyclic then
\[
\hmax (H\curvearrowright Y) \ge \hmax (G\curvearrowright X),
\]
while if $G$ is a product of two countably infinite groups at least one of which is not locally finite then
\[
\hmax (H\curvearrowright Y) \ge \hinf (G\curvearrowright X).
\]
\end{theoremi}

If the actions of $G$ and $H$ above are genuinely free or if $G$ and $H$ are torsion-free,
then using the variational
principle (Theorem~10.35 in \cite{KerLi16}) and (in the case of torsion-free $G$ and $H$)
the main result of \cite{Mey16} one can also derive the first part of the above theorem
from Theorem~A of \cite{KerLi19},
or from \cite{Aus16} if $G$ and $H$ are in addition amenable and finitely generated.

In parallel with the p.m.p.\ setting, we define a continuous action of a countable sofic group
on a compact metrizable space to be {\it entropy regular}
if its maximum and infimum sofic topological entropies are equal,
and note that this occurs in the following situations:
\begin{enumerate}
\item the group is amenable, in which case
the sofic topological entropy is equal to the amenable topological entropy \cite{KerLi13},

\item the action is a shift action $G\curvearrowright X^G$
where $X$ is a compact metrizable space
\cite[Proposition~10.28]{KerLi16},

\item the action
is an algebraic action of the form $G\curvearrowright \widehat{(\Zb G)^n / (\Zb G)^n A}$
where $A\in M_n (\Zb G)$ is injective as an operator on $\ell^2 (G)^{\oplus n}$ \cite{Hay16}.
\end{enumerate}
From Theorem~\ref{T-cts 1} we immediatety obtain:

\begin{theoremi}\label{T-cts 2}
Let $G$ and $H$ be countable sofic groups each of which either
\begin{enumerate}
\item[(a)] contains a w-normal amenable subgroup which is neither locally finite nor virtually cyclic, or

\item[(b)] is a product of two countably infinite sofic groups at least one of which is not locally finite.
\end{enumerate}
Let $G\curvearrowright X$ and $H\curvearrowright Y$ be
topologically free and entropy regular continuous actions on compact metrizable spaces, and
suppose that they are continuously orbit equivalent. Then
\[
h(H\curvearrowright Y) = h(G\curvearrowright X).
\]
\end{theoremi}

It was shown in \cite{ChuJia17,Coh18} that the finite-base shift actions
of a finitely generated group satisfy continuous cocycle superrigidity if and only if the group has one end
(a property that the groups in Theorem~\ref{T-cts 2} possess when they are finitely generated---see Example~1 in \cite{ChuJia17}).
As observed in \cite{ChuJia17}, this implies, in conjunction with a theorem from \cite{Li18},
that if a finitely generated group is torsion-free and amenable
then each of its shift actions with finite base is continuous orbit equivalence superrigid.
Whether such
superrigidity ever occurs in the nonamenable setting appears however to be unknown.

We begin the main body of the paper in Section~\ref{S-preliminaries} by setting up
general notation and reviewing terminology concerning continuous and bounded
orbit equivalence and sofic entropy.
In Section~\ref{SS-sofic SC} we define properties $\sS$-SC and sofic SC
for groups, p.m.p.\ actions, and continuous actions on compact metrizable spaces.
In Section~\ref{SS-without} we determine that a countable group fails
to have property sofic SC if it is locally finite or finitely generated and virtually free.
In Section~\ref{SS-SC} we verify that, for free p.m.p.\ actions, property SC implies
property sofic SC, and then use this in conjunction with \cite{KerLi19}
to show that (i) for countable amenable groups property sofic SC is equivalent to
the group being neither locally finite nor virtually cyclic,
and (ii) if a countable group has a w-normal subgroup which is amenable
but neither locally finite nor virtually cyclic then the group has property sofic SC.
In Section~\ref{SS-normal} we prove that if a w-normal subgroup has property sofic SC
then so does the ambient group, while
in Section~\ref{SS-products} we determine that the product of two countably infinite groups
has property $\sS$-SC, where $\sS$ is the collection of product sofic approximations,
if and only if at least one of the factors is not locally finite.
Sections~\ref{SS-invariant cts} and \ref{SS-invariant bounded} show property sofic SC to be an invariant
of continuous orbit equivalence for topologically free continuous actions
on compact metrizable spaces and of bounded orbit equivalence for free p.m.p.\ actions.
Section~\ref{S-continuous OE} is devoted to the proof
of Theorem~\ref{T-continuous OE}, which together with
Theorem~\ref{T-w-normal amenable} and Proposition~\ref{P-product} gives Theorem~\ref{T-cts 1}.
Finally, in Section~\ref{S-bounded OE} we establish Theorem~\ref{T-bounded OE},
which together with Proposition~\ref{P-product} yields Theorem~\ref{T-measure 1}.

\medskip

\noindent{\it Acknowledgements.}
The first author was partially supported by NSF grant DMS-1800633 and was affiliated with
Texas A\&M University during the course of the work.
Preliminary stages were carried out during his six-month stay at the ENS de Lyon in 2017-2018,
during which time he held ENS and CNRS visiting professorships and was supported by Labex MILYON/ANR-10-LABX-0070. 
The second author was partially supported by NSF grants DMS-1600717 and DMS-1900746.
We thank the referee for helpful comments and suggestions.

\section{Preliminaries}\label{S-preliminaries}

\subsection{Basic notation and terminology}

Throughout the paper $G$ and $H$ are countable discrete groups,
with identity elements $e_G$ and $e_H$.
We write $\cF(G)$ for the collection of all nonempty
finite subsets of $G$, and $\overline{\cF}(G)$ for the collection
of symmetric finite subsets of $G$ containing $e_G$.
For a nonempty finite set $V$, the algebra of all subsets of $V$ is denoted by $\Pb_V$,
the group of all permutations of $V$ by $\Sym (V)$,
and the uniform probability measure on $V$ by $\m$.

Given a property P, a group is said to be {\it virtually P}
if it has a subgroup of finite index with property P, and {\it locally P} if each of its
finitely generated subgroups has property P.

A {\it standard probability space} is a standard Borel space (i.e., a Polish space with its Borel $\sigma$-algebra)
equipped with a probability measure. Partitions of such a space are always understood to be Borel.
A {\it p.m.p.\ (probability-measure preserving) action} of $G$ is an action $G\curvearrowright (X,\mu )$
of $G$ on a standard probability space by measure-preserving transformations.
Such an action is {\it free} if the
set $X_0$ of all $x\in X$ such that $sx \neq x$ for all $s\in G\setminus \{ e_G \}$ has measure one.
Two p.m.p.\ actions $G\curvearrowright (X,\mu )$ and $G\curvearrowright (Y,\nu )$
are {\it measure conjugate} if there exist $G$-invariant conull sets
$X_0\subseteq X$ and $Y_0\subseteq Y$
and a $G$-equivariant measure isomorphism $X_0 \to Y_0$.

A continuous action $G\curvearrowright X$ on a compact metrizable space is
said to be {\it topologically free}
if the $G_\delta$ set of all $x\in X$ such that $sx \neq x$ for all $s\in G\setminus \{ e_G \}$ is dense.
It is {\it uniquely ergodic} if there is a unique $G$-invariant Borel probability measure on $X$.
By the Riesz representation theorem this is equivalent to the existence of a unique $G$-invariant state
(i.e., unital positive linear functional)
for the induced action of $G$ on the C$^*$-algebra $C(X)$ of continuous functions on $X$
given by $(gf)(x) = f(g^{-1} x)$ for all $g\in G$, $f\in C(X)$, and $x\in X$.

A p.m.p.\ action $G\curvearrowright (X,\mu )$ is {\it uniquely ergodic} if there is a unique state
(or {\it mean} as it is also called in this setting) on $L^\infty (X,\mu )$ which is invariant for the action of $G$ given by
$(gf)(x) = f(g^{-1} x)$ for all $g\in G$, $f\in L^\infty (X,\mu )$, and $x\in X$.
By Gelfand theory, this is equivalent to the unique ergodicity,
in the topological-dynamical sense above, of the induced action of $G$ on the spectrum of $L^\infty (X,\mu )$.

\subsection{Continuous orbit equivalence}

We say that two continuous actions $G\curvearrowright X$ and $H\curvearrowright Y$ on compact metrizable spaces
are {\it continuously orbit equivalent} if there exist a homeomorphism $\Phi : X\to Y$ and
continuous maps $\kappa : G\times X \to H$ and $\lambda : H\times Y \to G$ such that
\begin{align*}
\Phi (gx) &= \kappa (g,x) \Phi (x) ,\\
\Phi^{-1} (ty) &= \lambda (t,y) \Phi^{-1} (y)
\end{align*}
for all $g\in G$, $x\in X$, $t\in H$, and $y\in Y$.
Such a $\Phi$ is called a {\it continuous orbit equivalence}.

If the action $H\curvearrowright Y$ is topologically free
then the continuity of $\Phi$ implies
that the map $\kappa$ is uniquely determined by the first line of the above display and satisfies
the {\it cocycle identity}
\begin{align*}
\kappa (fg,x) = \kappa (f,gx) \kappa (g,x)
\end{align*}
for $f,g\in G$ and $x\in X$. In the case that both $G\curvearrowright X$ and $H\curvearrowright Y$
are topologically free we have
\begin{align*}
\lambda (\kappa (g,x),\Phi (x)) = g
\end{align*}
for all $g\in G$ and $x\in X$, and $\lambda$ is uniquely determined by this identity.

\subsection{Bounded orbit equivalence}\label{SS-bounded}

Two free p.m.p.\ actions $G\curvearrowright (X,\mu )$ and
$H\curvearrowright (Y,\nu )$
are {\it orbit equivalent} if there exist a $G$-invariant conull set $X_0 \subseteq X$,
an $H$-invariant conull set $Y_0 \subseteq Y$, and a measure isomorphism
$\Psi : X_0 \to Y_0$ such that $\Psi (Gx) = H\Psi (x)$ for all $x\in X_0$.
Such a $\Psi$ is called an {\it orbit equivalence}.
Associated to $\Psi$ are the cocycles
$\kappa : G\times X_0 \to H$ and $\lambda : H\times Y_0 \to G$ determined (up to null sets,
in accord with our definition of freeness) by
\begin{align*}
\Psi (gx) &= \kappa (g,x)\Psi (x) , \\
\Psi^{-1} (ty) &= \lambda (t,y) \Psi^{-1} (y)
\end{align*}
for all $g\in G$, $x\in X_0$, $t\in H$, and $y\in Y_0$.
We say that the cocycle $\kappa$ is {\it bounded} if $\kappa (g,X_0 )$ is finite
for every $g\in G$, and define boundedness for $\lambda$ likewise.
If $X_0$, $Y_0$, and $\Psi$ can be chosen so that $\kappa$ and $\lambda$ are both
bounded, then we say that the actions are {\it boundedly orbit equivalent}, and refer to $\Psi$ as a
{\it bounded orbit equivalence}.

\subsection{Sofic approximations}

Given a nonempty finite set $V$ we define on $V^V$ the normalized Hamming distance
\[
\rho_\Hamm (T,S) = \frac{1}{|V|} |\{ v\in V : Tv \neq Sv \} | .
\]
A {\it sofic approximation} for $G$ is a (not necessarily multiplicative) map
$\sigma : G\to\Sym (V)$ for some nonempty finite set $V$.
Given a finite set $F\subseteq G$ and a $\delta > 0$, we say that such a $\sigma$
is an {\it $(F,\delta )$-approximation} if
\begin{enumerate}
\item $\rho_\Hamm (\sigma_{st} , \sigma_s \sigma_t ) \leq \delta$ for all $s,t\in F$, and

\item $\rho_\Hamm (\sigma_s , \sigma_t ) \geq 1-\delta$ for all distinct $s,t\in F$.
\end{enumerate}
A {\it sofic approximation sequence} for $G$ is a sequence
$\Sigma = \{ \sigma_k : G\to\Sym (V_k ) \}_{k=1}^\infty$
of sofic approximations for $G$ such that
for every finite set $F\subseteq G$ and $\delta > 0$ there exists a $k_0 \in\Nb$ such that
$\sigma_k$ is an $(F,\delta )$-approximation for every $k\geq k_0$.
A sofic approximation $\sigma : G\to\Sym (V)$ is said to be {\it good enough} if
it is an $(F,\delta )$-approximation for some finite set $F\subseteq G$ and $\delta > 0$
and this condition is sufficient for the purpose at hand.

The group $G$ is {\it sofic} if it admits a sofic approximation sequence,
which is the case for instance if $G$ is amenable or residually finite.
It is not known whether nonsofic groups exist.

Given a sofic approximation $\sigma : G\to\Sym (V)$ and a set $A\subseteq G$, we define
an {\it $A$-path} to be a finite tuple $(v_0 , v_1 , \dots , v_n )$ of points in $V$ such that for every
$i=1,\dots , n$ there is a $g\in A$ for which $v_i = \sigma_g v_{i-1}$. The integer $n$
is the {\it length} of the path, the points $v_0 , \dots , v_n$ its {\it vertices},
and $v_0$ and $v_n$ its {\it endpoints}.
When $n=1$ we also speak of an {\it $A$-edge}.
For $r\in\Nb$, we say that a set $W\subseteq V$ is {\it $(A,r)$-separated}
if $\sigma_{A^r} v \cap \sigma_{A^r} w = \emptyset$ for all distinct $v,w\in W$.

\subsection{Sofic topological entropy}\label{SS-sofic topological entropy}

Let $G\curvearrowright X$ be a continuous action on a compact metrizable space.
Let $d$ be a compatible metric on $X$. Let $F$ be a finite subset of $G$ and $\delta > 0$.
Let $\sigma : G\to \Sym (V)$ be a sofic approximation for $G$.
On the set of maps $V \to X$ define the pseudometrics
\begin{align*}
d_2 (\varphi ,\psi ) &= \bigg( \frac{1}{|V|} \sum_{v\in V} d(\varphi (v),\psi (v))^2 \bigg)^{1/2} , \\
d_\infty (\varphi ,\psi ) &= \max_{v\in V} d(\varphi (v),\psi (v)) .
\end{align*}
Define $\Map_d (F,\delta ,\sigma )$ to be the set of all maps $\varphi : V\to X$
such that $d_2 (\varphi \sigma_g,g\varphi ) \leq\delta$ for all $g\in F$.
For a pseudometric space $(\Omega ,\rho )$ and $\eps > 0$ we write $N_\eps (\Omega ,\rho )$
for the maximum cardinality of a subset $\Omega_0$ of $\Omega$ which is $(\rho ,\eps )$-separated
in the sense that $\rho (\omega_1 ,\omega_2 ) \geq \eps$ for all distinct $\omega_1 ,\omega_2 \in \Omega_0$.

Let $\Sigma = \{ \sigma_k : G\to\Sym (V_k ) \}_{k=1}^\infty$ be a sofic
approximation sequence for $G$. For $\eps > 0$ we set
\begin{align*}
h_{\Sigma ,\infty}^\eps (G\curvearrowright X) &= \inf_F \inf_{\delta > 0}
\limsup_{k\to\infty} \frac{1}{|V_k|} \log N_\eps (\Map_d (F,\delta ,\sigma_k ) , d_\infty ) ,\\
h_{\Sigma ,2}^\eps (G\curvearrowright X) &= \inf_F \inf_{\delta > 0}
\limsup_{k\to\infty} \frac{1}{|V_k|} \log N_\eps (\Map_d (F,\delta ,\sigma_k ) , d_2 ) ,
\end{align*}
where the first infimum in each case is over all finite sets $F\subseteq G$.
The {\it sofic topological entropy} of the action $G\curvearrowright X$ with respect to $\Sigma$
is then defined by
\begin{align*}
h_\Sigma (G\curvearrowright X) = \sup_{\eps > 0} h_{\Sigma ,\infty}^\eps (G\curvearrowright X) .
\end{align*}
This quantity does not depend on the choice of compatible metric $d$, as is readily seen, and
by Proposition~10.23 of \cite{KerLi16} we can also compute it using separation with respect to $d_2$, i.e.,
\begin{align*}
h_\Sigma (G\curvearrowright X) = \sup_{\eps > 0} h_{\Sigma ,2}^\eps (G\curvearrowright X) .
\end{align*}

We define the {\it maximum} and {\it infimum sofic topological entropies} of $G\curvearrowright X$ by
\begin{align*}
\hmax (G\curvearrowright X) &= \max_\Sigma h_\Sigma (G\curvearrowright X) ,\\
\hinf (G\curvearrowright X) &= \inf_\Sigma h_\Sigma (G\curvearrowright X) ,
\end{align*}
where $\Sigma$ ranges in each case over all sofic approximation sequences for $G$
(when $G$ is nonsofic we interpret these quantities to be $-\infty$).
It is a straightforward
exercise to show that the maximum does indeed exist (we do not know however whether the infimum is always realized).
Note that $-\infty$ is a possible value for
$h_\Sigma (G\curvearrowright X)$, and so if it occurs for some $\Sigma$ then
$\hinf (G\curvearrowright X) = -\infty$, and if it occurs for all $\Sigma$ then $\hmax (G\curvearrowright X) = -\infty$.
The action $G\curvearrowright X$ is {\it entropy regular}
if its maximum and infimum sofic topological entropies are equal, i.e., the sofic topological entropy
does not depend on the choice of sofic approximation sequence.

\subsection{Sofic measure entropy}\label{SS-sofic measure entropy}

Let $G\curvearrowright (X,\mu )$ be a p.m.p.\ action.
Let $\sC$ be a finite Borel partition of $X$, $F$ a finite subset of $G$ containing $e_G$, and $\delta > 0$.
Write $\alg (\sC )$ for the algebra generated by $\sC$, which consists of all unions of members of $\sC$,
and write $\sC_F$ for the join $\bigvee_{s\in F} s\sC$.
Let $\sigma : G\to \Sym(V)$ be a sofic approximation for $G$. Write $\Hom_\mu (\sC ,F,\delta ,\sigma )$ for the
set of all homomorphisms $\varphi : \alg (\sC_F )\to \Pb_V$ satisfying
\begin{enumerate}
\item $\sum_{A\in\sC} \m (\sigma_g \varphi (A) \Delta \varphi (gA)) < \delta$ for all $g\in F$, and

\item $\sum_{A\in\sC_F} |\m (\varphi (A)) - \mu (A)| < \delta$.
\end{enumerate}
For a finite Borel partition $\sP\leq\sC$ we write $|\Hom_\mu (\sC ,F,\delta ,\sigma )|_\sP$ for the cardinality
of the set of restrictions of elements of $\Hom_\mu (\sC ,F,\delta ,\sigma )$ to $\sP$.

Given a sofic approximation sequence
$\Sigma = \{ \sigma_k : G\to\Sym (V_k ) \}_{k=1}^\infty$ for $G$, we define
the {\it sofic measure entropy} of the action $G\curvearrowright (X,\mu )$ with respect to $\Sigma$ by
\begin{align*}
h_{\Sigma ,\mu} (G\curvearrowright X) = \sup_\sP \inf_{\sC\geq\sP} \inf_F \inf_{\delta > 0}
\limsup_{k\to\infty} \frac{1}{|V_k|} \log |\Hom_\mu (\sC ,F,\delta ,\sigma_k ) |_\sP ,
\end{align*}
where the supremum is over all finite partitions $\sP$ of $X$,
the first infimum is over all finite partitions $\sC$ of $X$ refining $\sP$, and the second infimum
is over all finite sets $F\subseteq G$ containing $e_G$.

As in the topological case, one can check that there is a maximum among the
quantities $h_{\Sigma ,\mu} (G\curvearrowright X)$ over all sofic approximation sequences $\Sigma$ for $G$,
where $-\infty$ is included as a possible value.
The {\it maximum} and {\it infimum sofic measure entropies} of $G\curvearrowright (X,\mu )$ are then defined by
\begin{align*}
\hmax_\mu (G\curvearrowright X) &= \max_\Sigma h_{\Sigma ,\mu} (G\curvearrowright X) , \\
\hinf_\mu (G\curvearrowright X) &= \inf_\Sigma h_{\Sigma ,\mu} (G\curvearrowright X) ,
\end{align*}
where $\Sigma$ ranges in each case over all sofic approximation sequences for $G$.
When $G$ is nonsofic these quantities are interpreted to be $-\infty$.
The action is {\it entropy regular}
if its maximum and infimum sofic measure entropies are equal, i.e., the sofic measure entropy
does not depend on the choice of sofic approximation sequence.

\section{Properties $\sS$-SC and sofic SC} \label{S-sofic SC}

\subsection{Definitions of properties $\sS$-SC and sofic SC}\label{SS-sofic SC}

Write $\sS_G$ for the collection of all sofic approximations for $G$.

Let $\sS$ be any collection of sofic approximations for $G$.

\begin{definition}\label{D-CSC 1}
We say that the group $G$ has {\it property $\sS$-SC} (or {\it property sofic SC} if $\sS = \sS_G$)
if for any function $\Upsilon: \cF(G)\rightarrow [0, \infty)$
there exists an  $S\in \overline{\cF}(G)$ such that for any  $T\in \overline{\cF}(G)$ there are $C, n\in \Nb$,
and $S_1, \dots, S_n\in \overline{\cF}(G)$ such that for every good enough sofic approximation
$\pi: G\rightarrow \Sym(V)$ in $\sS$ there are subsets $W$ and $\cV_j$ of $V$ for $1\le j\le n$
satisfying the following conditions:
\begin{enumerate}
\item $\sum_{j=1}^n \Upsilon(S_j)\m(\cV_j)\le 1$,
\item  $\bigcup_{g\in S}\pi_gW=V$,
\item if $w_1, w_2\in W$ satisfy $\pi_gw_1=w_2$ for some $g\in T$ then $w_1$ and $w_2$ are connected by a path of length at most $C$ in which each edge is of the form $(v, \pi_hv)$ for some $1\le j\le n$, $h\in S_j$, and
$v\in \cV_j$ with $\pi_hv\in \cV_j$.
\end{enumerate}
\end{definition}

\begin{definition}\label{D-CSC 2}
We say that a continuous action $G\curvearrowright X$ on a compact metrizable space $X$ with compatible metric $d$ has
{\it property $\sS$-SC} (or {\it property sofic SC} if $\sS = \sS_G$) if for any function $\Upsilon: \cF(G)\rightarrow [0, \infty)$ there exists an  $S\in \overline{\cF}(G)$ such that for any $T\in \overline{\cF}(G)$ there are $C, n\in \Nb$,  $S_1, \dots, S_n\in \overline{\cF}(G)$, $F^\sharp\in \cF(G)$, and $\delta^\sharp>0$ such that for every good enough sofic approximation $\pi: G\rightarrow \Sym(V)$ in $\sS$ with $\Map_d(F^\sharp, \delta^\sharp, \pi)\neq \emptyset$ there are $W$ and $\cV_j$ for $1\le j\le n$ as in Definition~\ref{D-CSC 1}. By \cite[Lemma 10.24]{KerLi16} this does not depend on the choice of $d$.
\end{definition}

\begin{definition}\label{D-CSC 3}
We say that a p.m.p.\ action $G\curvearrowright (X,\mu )$ has {\it property $\sS$-SC}
(or {\it property sofic SC} if $\sS = \sS_G$)
if for any function $\Upsilon: \cF(G)\rightarrow [0, \infty)$ there exists an $S\in \overline{\cF}(G)$  such that for any  $T\in \overline{\cF}(G)$ there are $C, n\in \Nb$,  $S_1, \dots, S_n\in \overline{\cF}(G)$, a finite Borel partition $\sC^\sharp$ of $X$, an $F^\sharp\in \cF(G)$ containing $e_G$, and a $\delta^\sharp>0$ such that for every good enough sofic approximation $\pi: G\rightarrow \Sym(V)$ in $\sS$ with $\Hom_\mu(\sC^\sharp, F^\sharp, \delta^\sharp, \pi)\neq \emptyset$ there are $W$ and $\cV_j$ for $1\le j\le n$ as in Definition~\ref{D-CSC 1}.
\end{definition}

The following proposition shows that, when $G$ is finitely generated, in Definition~\ref{D-CSC 1}
we can fix $n=1$ and take $S_1$ to be any symmetric finite generating subset of $G$ containing $e_G$, but with the price that $\bigcup_{g\in S} \pi_g W$ is only most of $V$ instead of the whole of $V$.

\begin{proposition} \label{P-sofic SC for fg}
Suppose that $G$ is finitely generated. Let $A$ be a generating set for $G$ in $\overline{\cF}(G)$. Then $G$ has property
$\sS$-SC if and only if for any $\varepsilon>0$ there exists an  $S\in \overline{\cF}(G)$  such that for any $T\in \overline{\cF}(G)$ and $\delta>0$ there is  a $C\in \Nb$ such that for any good enough sofic approximation $\pi: G\rightarrow \Sym(V)$
in $\sS$ there are subsets $W$ and $\cV$ of $V$ satisfying the following conditions:
\begin{enumerate}
\item $\m(\cV)\le \varepsilon $,
\item $\m(\bigcup_{g\in S}\pi_gW)\ge 1-\delta$,
\item if $w_1, w_2\in W$ satisfy $\pi_gw_1=w_2$ for some $g\in T$ then $w_1$ and $w_2$ are connected by an $A$-path of length at most $C$ whose vertices all lie in $\cV$.
\end{enumerate}
\end{proposition}

\begin{proof}
Denote by $\ell_A$ the word length function on $G$ associated to $A$.

Suppose first that $G$ has property $\sS$-SC. Let $\varepsilon>0$. Define $\Upsilon: \cF(G)\rightarrow [0, \infty)$ by $\Upsilon(F)=\varepsilon^{-1}|A|^{\max_{g\in F}\ell_A(g)}$. Then there is an $S\in \overline{\cF}(G)$ witnessing property $\sS$-SC. Let $T\in \overline{\cF}(G)$ and $\delta>0$. Then we have $C, n, S_1, \dots, S_n$ as given by Definition~\ref{D-CSC 1}. Set $m=\max_{1\le j\le n}\max_{g\in S_j}\ell_A(g)$. Let $\pi: G\rightarrow \Sym(V)$ be a good enough sofic approximation for $G$ in $\sS$. Then we have $W$ and $\cV_1, \dots, \cV_n$ satisfying conditions (i)-(iii) in Definition~\ref{D-CSC 1}. Denote by $V'$ the set of all $v\in V$ satisfying $\pi_{gh}v=\pi_g\pi_h v$ for all $g, h\in A^{mC}$.
When $\pi$ is a good enough sofic approximation, we have $\m(V\setminus V')\le \delta/|S|$. Set $W'=W\cap V'$.
For each $1\le j\le n$, set $m_j=\max_{g\in S_j}\ell_A(g)$ and $\cV_j^\dag=\bigcup_{g\in A^{m_j}}\pi_g\cV_j$. Set $\cV=\bigcup_{j=1}^n\cV_j^\dag$. Then
\[
\m(\cV)\le \sum_{j=1}^n\m(\cV_j^\dag)\le \sum_{j=1}^n|A^{m_j}|\cdot \m(\cV_j)\le \varepsilon\sum_{j=1}^n\Upsilon(S_j)\m(\cV_j)\le \varepsilon,
\]
verifying condition (i) in the proposition statement.
Note also that
\[\m\bigg(\bigcup_{g\in S}\pi_gW'\bigg)\ge \m\bigg(\bigcup_{g\in S}\pi_g W\bigg)-|S|\cdot \m(V\setminus V')\ge 1-\delta,\]
which verifies condition (ii) in the proposition statement.
Let $g\in T$ and $w_1, w_2\in W'$ be such that $\pi_g w_1=w_2$. Then $w_1$ and $w_2$ are connected by a path of length at most $C$ in which each edge is an $S_j$-edge with both endpoints in $\cV_j$ for some $1\le j\le n$. It is easily checked that the
endpoints of such an edge are connected by an $A$-path of length at most $m_j$ with all vertices in $\cV_j^\dag$. Thus $w_1$ and $w_2$ are connected by an $A$-path of length at most $Cm$ with all vertices  in $\cV$, verifying condition (iii)
in the proposition statement. This proves the ``only if'' part.

To prove the ``if'' part, suppose that $G$ satisfies the condition in the statement of the proposition. Let $\Upsilon$ be a function $\cF(G)\rightarrow [0, \infty)$.
Take $0<\varepsilon<1/(2\Upsilon(A))$.
Then we have an $S$ as in the statement of the proposition.   Let $T\in \overline{\cF}(G)$. Take $0< \delta<1/(6|T|\Upsilon(T))$. Then we have a $C$ as in the statement of the proposition. Set $n=2$, $S_1=A$, and $S_2=T$. Let $\pi: G\rightarrow \Sym(V)$ be a good enough sofic approximation for $G$ in $\sS$. Then we have $W$ and $\cV$ as in the statement of the proposition.
Set $W'=W\cup \pi_{e_G}^{-1}(V\setminus \bigcup_{g\in S}\pi_g W)$. Then $\bigcup_{g\in S}\pi_gW'=V$, verifying condition (ii) in Definition~\ref{D-CSC 1}.
Set $\cV_1=\cV$ and $\cV_2=\bigcup_{g\in T}((W'\setminus W)\cup \pi_g(W'\setminus W)\cup \pi_g^{-1}(W'\setminus W))$.
Then
\[
\m(\cV_2)\le (2|T|+1)\m(W'\setminus W)\le 3|T|\delta,
\]
and hence
\[
\Upsilon(S_1)\m(\cV_1)+\Upsilon(S_2)\m(\cV_2)\le \Upsilon(A)\varepsilon+3\Upsilon(T)|T|\delta< \frac12 + \frac12 = 1,
\]
which verifies condition (i) in Definition~\ref{D-CSC 1}.
Let $g\in T$ and $w_1, w_2\in W'$ be such that $\pi_gw_1=w_2$. If $w_1\not\in W$ or $w_2\not\in W$, then $(w_1, w_2)$ is an $S_2$-edge with both endpoints in $\cV_2$. If $w_1, w_2\in W$, then $w_1$ and $w_2$ are connected by  an $S_1$-path of length at most $C$ such that all vertices of this path lie in $\cV=\cV_1$, yielding condition (iii) in Definition~\ref{D-CSC 1}.
\end{proof}

\subsection{Groups without property sofic SC}\label{SS-without}

Let $\sS$ be a collection of sofic approximations for $G$ which contains
arbitrarily good sofic approximations (or, equivalently, which contains a sofic approximation sequence).
In Propositions~\ref{P-locally finite group} and \ref{P-virtually free group}
we identify two classes of groups which fail to have property $\sS$-SC, and
in particular fail to have property sofic SC.

\begin{lemma} \label{L-finite group}
Suppose that $G$ is finite. Then
$G$ does not have property $\sS$-SC.
\end{lemma}

\begin{proof}
Suppose to the contrary that $G$ has property $\sS$-SC.
Define $\Upsilon: \cF(G)\rightarrow [0, \infty)$ by
$\Upsilon(F)=4|G|$ for all $F\in \cF(G)$.
Then there is some  $S\in \overline{\cF}(G)$  satisfying the conditions in Definition~\ref{D-CSC 1}. Put  $T=\{e_G\}$. Then there are  $C, n\in \Nb$ and $S_1, \dots, S_n\in \overline{\cF}(G)$  satisfying the conditions in Definition~\ref{D-CSC 1}.

Let $\pi: G\rightarrow \Sym(V)$ be a good enough sofic approximation in $\sS$ so that
there are subsets $W$ and $\cV_1, \dots, \cV_n$ of $V$ satisfying conditions (i)-(iii) in Definition~\ref{D-CSC 1} and
also so that $\m(U)\le 1/(4|G|)$ where $U$ consists of all
$v\in V$ satisfying $\pi_{e_G}v\neq v$. Seeing that $W\subseteq U\cup \bigcup_{j=1}^n\cV_j$, we have
\begin{align*}
\m(W)\le \m(U)+\m\bigg(\bigcup_{j=1}^n\cV_j\bigg)\le \frac{1}{4|G|}+\sum_{j=1}^n\m(\cV_j)=\frac{1}{4|G|}+\frac{1}{4|G|}\sum_{j=1}^n\Upsilon(S_j)\m(\cV_j)\le \frac{1}{2|G|}.
\end{align*}
Thus
\[ 1= \m\bigg(\bigcup_{g\in S}\pi_gW\bigg)\le |S|\m(W)\le |G|\m(W)\le \frac12 ,\]
a contradiction.
\end{proof}

\begin{proposition} \label{P-locally finite group}
Suppose that $G$ is locally finite. Then $G$ does not have property $\sS$-SC.
\end{proposition}

\begin{proof}
Suppose to the contrary that $G$ has property $\sS$-SC.
Then $G$ must be infinite by Lemma~\ref{L-finite group}.
Take a strictly increasing sequence $\{G_k\}$ of finite subgroups of $G$ such that $G=\bigcup_{k\in \Nb}G_k$.  For each $F\in \cF(G)$, denote by $\Phi(F)$ the smallest $k\in \Nb$ satisfying $F\subseteq G_k$. Define $\Upsilon: \cF(G)\rightarrow [0, \infty)$ by
$\Upsilon(F)=3|G_{\Phi(F)}|$.
Then there is some  $S\in \overline{\cF}(G)$  satisfying the conditions in Definition~\ref{D-CSC 1}. Put $m=\Phi(S)$ and $T=G_{m+1}\in \overline{\cF}(G)$. Then there are  $C, n\in \Nb$ and $S_1, \dots, S_n\in \overline{\cF}(G)$ satisfying the conditions in Definition~\ref{D-CSC 1}.
Put $M=\max\{\max_{1\le j\le n}\Phi(S_j), m+1\}$.

Let $\pi: G\rightarrow \Sym(V)$ be a good enough sofic approximation in $\sS$ so that there is a set $V_1\subseteq V$ satisfying the following conditions:
\begin{enumerate}
\item $\pi_g\pi_hv=\pi_{gh}v$ for all $g, h\in G_M$ and $v\in V_1$,
\item $\pi_gv\neq \pi_hv$ for all $v\in V_1$ and distinct $g, h\in G_M$,
\item $\pi_gV_1=V_1$ for all $g\in G_M$,
\item $\m(V_1)\ge 1/2$.
\end{enumerate}
Then $G_M$ acts on $V_1$ via $\pi$.
Denote by $\sP$ the  partition  of $V_1$ into $G_{m+1}$-orbits.

By assumption, when $\pi$ is a good enough sofic approximation we can find subsets $W$ and $\cV_1, \dots, \cV_n$ of $V$ satisfying conditions (i)-(iii) in Definition~\ref{D-CSC 1}. Note that $V=\bigcup_{g\in S}\pi_gW=\bigcup_{g\in G_m}\pi_gW$,
which implies that for every member $P$ of $\sP$ the intersection $P\cap W$ is not contained in a single $G_m$-orbit.

Set $\cV=\bigcup_{j=1}^n\pi_{G_{\Phi(S_j)}}(\cV_j\cap V_1)$. Then
\begin{gather}\label{E-third}
\m(\cV)\le \sum_{1\le j\le n}|G_{\Phi(S_j)}|\m(\cV_j\cap V_1)\le \sum_{1\le j\le n}|G_{\Phi(S_j)}|\m(\cV_j)= \frac{1}{3}\sum_{1\le j\le n}\Upsilon(S_j)\m(\cV_j)\le \frac{1}{3}.
\end{gather}

Now let $P\in \sP$ and $w_1\in P\cap W$. Then we can find some $w_2\in P\cap W$ such that $w_1$ and $w_2$ are in different $G_m$-orbits. We have $w_2=\pi_tw_1$ for some $t\in G_{m+1}\setminus G_m=T\setminus G_m$.
Thus we can find some $1\le l\le C$, $1\le j_1, \dots, j_l\le n$, $v_k\in \cV_{j_k}$ for $1\le k\le l$, and $g_k\in S_{j_k}$ for $1\le k\le l$  such that, setting $v_0=w_1$, we have $\pi_{g_k}v_{k-1}=v_k$ for all $1\le k\le l$ and $v_l=w_2$.
Then
\[
w_2=v_l=\pi_{g_l}\dots \pi_{g_1}v_0=\pi_{g_l\dots g_1}w_1,
\]
and hence $g_l\cdots g_1=t$. It follows that the elements $g_1, \dots, g_l$ cannot all lie  in $G_m$. Denote by $i$ the smallest $k$ satisfying $g_k\not\in G_m$. Then $v_i\in \pi_{G_{\Phi(S_{j_i})}}w_1$, and hence $w_1\in \pi_{G_{\Phi(S_{j_i})}}v_i$. Consequently,
\[
\pi_{G_m}w_1\subseteq \pi_{G_m}\pi_{G_{\Phi(S_{j_i})}}v_i=\pi_{G_{\Phi(S_{j_i})}}v_i\subseteq \pi_{G_{\Phi(S_{j_i})}}(\cV_{j_i}\cap V_1)\subseteq \cV.
\]
Therefore $V_1=\bigcup \sP=\pi_{G_m}(W\cap (\bigcup \sP))\subseteq \cV$, whence
$ \m(\cV)\ge \m(V_1)\ge 1/2$, contradicting (\ref{E-third}).
\end{proof}

\begin{proposition} \label{P-virtually free group}
Suppose that $G$ is finitely generated and virtually free.
Then $G$ does not have property $\sS$-SC.
\end{proposition}

\begin{proof}
By Lemma~\ref{L-finite group} we may assume that $G$ is infinite. Take a free subgroup $G_1$ of $G$ with finite index. Then $G_1$ is nontrivial and, by Schreier's lemma, finitely generated.
Take free generators $a_1, \dots, a_r$ for $G_1$. Set $A=\{a_1, \dots, a_r, a_1^{-1}, \dots, a_r^{-1}, e_G\}$.
Denote by $\ell$ the word length function on $G_1$ associated to $a_1, \dots, a_r, a_1^{-1}, \dots, a_r^{-1}$. For each $n\in \Nb$ denote by $B_n$ the set of elements $g$ in $G_1$ satisfying $\ell(g)\le n$. Take a subset $H$ of $G$ containing $e_G$ such that $G$ is the disjoint union of the sets $hG_1$ for $h\in H$. Set $D=H\cup A$.
For each $F\in \cF(G)$, denote by $\Psi(F)$ the smallest $n\in \Nb$ satisfying $F\subseteq HB_n$, and set $F'=HB_{\Psi(F)}$.

For any $g\in G$ and $h\in H$ we can write $gh$ uniquely as $bd$ with $b\in H$ and $d\in G_1$, and using
this factorization we set $R$ to be the maximum value of $\ell(d)$ over all $g\in D$ and  $h\in H$.
Define $\Upsilon: \cF(G)\rightarrow [0, \infty)$ by $\Upsilon(F)=3|HB_{R}|\cdot |F'|$.

Suppose that $G$ has property $\sS$-SC. Then there is some $S\in \overline{\cF}(G)$ satisfying the conditions in Definition~\ref{D-CSC 1}. Put $m=\Psi(S)$, $N=m+1$, and $T=S\{a_1^{2N}, e_G, a_1^{-2N}\}S\in \overline{\cF}(G)$. Then there are  $C, n\in \Nb$ and  $S_1, \dots, S_n\in \overline{\cF}(G)$  satisfying the conditions in Definition~\ref{D-CSC 1}.  Put $C'=C\max_{1\le j\le n}(1+\Psi(S_j))\in \Nb$ and  $U=((HA)^{C'}HB_{2N})\cup ((HA)^{C'}HB_{2N})^{-1}\in \overline{\cF}(G)$.

Let $\pi: G\rightarrow \Sym(V)$ be a good enough sofic approximation in $\sS$ so that there are subsets $W$ and $\cV_1, \dots, \cV_n$ of $V$ satisfying conditions (i)-(iii) in Definition~\ref{D-CSC 1} and a set $V'\subseteq V$ satisfying the following conditions:
\begin{enumerate}
\item $\pi_g\pi_hv=\pi_{gh}v$ for all $g, h\in U^{10}$ and $v\in V'$,
\item $\pi_gv\neq \pi_hv$ for all $v\in V'$ and distinct $g, h\in U$,
\item $\m(V')\ge 1/2$.
\end{enumerate}
For each $1\le j\le n$, set
$\cV_j'=\bigcup_{g\in S_j'}\pi_g\cV_j.$
Set $\cV=\bigcup_{j=1}^n\cV_j'$.

Let $v\in V'$. We have $\pi_{a_1^{N}}v=\pi_{h_1}w_1$ and $\pi_{a_1^{-N}}v=\pi_{h_2}w_2$ for some $h_1, h_2\in S$ and $w_1, w_2\in W$. Then
\[
\pi_{h_2^{-1}a_1^{-2N}h_1}w_1=\pi_{h_2}^{-1}\pi_{a_1^{-N}}\pi_{a_1^{N}}^{-1}\pi_{h_1}w_1=\pi_{h_2}^{-1}\pi_{a_1^{-N}}v=w_2.
\]
By assumption we can find a path from $w_1$ to $w_2$ of length at most $C$ in which each edge is an $S_j$-edge with both endpoints in $\cV_j$ for some $1\le j\le n$. Replacing each such edge by a $D$-path of length at most $1+\Psi(S_j)$ and with all vertices in $\cV_j'$, we  find a $D$-path from $w_1$ to $w_2$ of length at most $C'$ such that all vertices are in $\cV$.
Thus we get some $1\le l\le C'$,  $v_k\in \cV$ for $1\le k\le l$, and $g_k\in D$ for $1\le k\le l$  such that, setting $v_0=w_1$, we have $\pi_{g_k}v_{k-1}=v_k$ for all $1\le k\le l$ and $v_l=w_2$.
Then
\[
\pi_{h_2^{-1}a_1^{-N}}v=w_2=\pi_{g_l}\dots \pi_{g_1}w_1=\pi_{g_l\dots g_1h_1^{-1}a_1^{N}}v.
\]
Since $h_2^{-1}a_1^{-N}$ and $g_l\dots g_1h_1^{-1}a_1^{N}$ belong to $U$, we conclude that $h_2^{-1}a_1^{-N}=g_l\dots g_1h_1^{-1}a_1^{N}$.
 Set $t_j=g_j\dots g_1h_1^{-1}a_1^{N}\in U$ for $0\le j\le l$.
 We can write each $t_j$ uniquely as $b_jd_j$ for some $b_j\in H$ and $d_j\in G_1$. Then we have
\[
\ell(d_jd_{j-1}^{-1})\le R
\]
for all $1\le j\le l$.
Consider the path $p$ in $G_1$ from $d_0$ to $d_l$ defined by concatenating the geodesic from $d_{j-1}$ to $d_j$ for all $1\le j\le l$, where we endow $G_1$ with the right invariant metric induced from $\ell$. Note that as reduced words $d_0$ and $d_l$ end with $a_1$ and $a_1^{-1}$ respectively. Thus $p$ passes through $e_G$. It follows that there is some $1\le i\le l$ with $\ell(d_{i})\le R$. Then $t_i\in HB_{R}$, whence
\[
v=\pi_{t_i}^{-1}v_i=\pi_{t_i^{-1}}v_i\in \bigcup_{g\in (HB_{R})^{-1}}\pi_g \cV.
\]
Therefore $V'\subseteq \bigcup_{g\in (HB_{R})^{-1}}\pi_g \cV$.

Now we get
\begin{gather*}
\frac12 \le \m(V')
\le \m \bigg(\bigcup_{g\in (HB_{R})^{-1}}\pi_g \cV \bigg)
\le |HB_{R}|\m(\cV) \hspace*{40mm} \\
\hspace*{40mm} \ \le |HB_{R}|\sum_{j=1}^n|S_j'|\m(\cV_j)
=\frac{1}{3}\sum_{j=1}^n\Upsilon(S_j)\m(\cV_j)\le \frac13 ,
\end{gather*}
a contradiction.
\end{proof}

\subsection{Groups with property sofic SC} \label{SS-SC}

In Theorems~\ref{T-compare SC} and \ref{T-w-normal amenable} below we will identify classes of groups
that have property sofic SC. This will rely on results from \cite{KerLi19} that we can access
via the connection to property SC established in Proposition~\ref{P-SC to sofic SC}.

\begin{definition} \label{D-SC}
Let $\fC$ be a class of free p.m.p.\ actions of a fixed infinite $G$.
We say that $\fC$ has {\it property SC} if for any function $\Upsilon: \cF(G)\rightarrow [0, \infty)$
there exists an $S\in \overline{\cF}(G)$ such that for any $T\in \overline{\cF}(G)$
there are $C, n\in \Nb$, and  $S_1, \dots, S_n\in \overline{\cF}(G)$  so that for any $G\curvearrowright (X, \mu)$ in $\fC$ there are Borel subsets $W$ and $\cV_j$ of $X$ for $1\le j\le n$  satisfying the following conditions:
\begin{enumerate}
\item $\sum_{j=1}^n \Upsilon(S_j)\mu(\cV_j)\le 1$,
\item $SW=X$,
\item if $w_1, w_2\in W$ satisfy $gw_1=w_2$ for some $g\in T$ then $w_1$ and $w_2$ are connected by a path of length at most $C$ in which each edge is an $S_j$-edge with both endpoints in $\cV_j$ for some $1\le j\le n$.
\end{enumerate}
We say that a p.m.p.\ action $G\curvearrowright (X, \mu)$ has {\it property SC} if the singleton class
containing it has property SC.
We say that $G$ itself has {\it property SC} if the class of all free p.m.p\ actions $G\curvearrowright (X, \mu)$
has property SC (note that freeness implies atomlessness of the measure since $G$ is infinite).
\end{definition}

\begin{remark} \label{R-SC no bound}
When $\fC$ consists of either a single free p.m.p.\ action or all free p.m.p.\ actions of a fixed $G$,
the existence of the bound $C$ is automatic, as explained in the paragraph following Proposition~3.5 in \cite{KerLi19}.
\end{remark}

\begin{proposition} \label{P-SC to sofic SC}
Suppose that $G$ is infinite and sofic. Let $G\curvearrowright (X, \mu)$ be a free p.m.p.\ action with property SC.
Then the action has property sofic SC.
\end{proposition}

\begin{proof}
We may assume, by passing to a suitable $G$-invariant conull subset of $X$, that the action of $G$ is genuinely free.
Let $\Upsilon$ be a function $\cF(G)\rightarrow [0, \infty)$.
Since $G\curvearrowright (X, \mu)$ has property SC, using the function $2\Upsilon$ we find an  $S\in \overline{\cF}(G)$  such that for any $T\in \overline{\cF}(G)$ there are $C, n\in \Nb$, $S_1, \dots, S_n\in \overline{\cF}(G)$, and Borel subsets $W$ and $\cV_k$ of $X$ for $1\le k\le n$ satisfying the following conditions:
\begin{enumerate}
\item $2\sum_{k=1}^n \Upsilon(S_k)\mu(\cV_k)\le 1$,
\item $SW=X$,
\item if $w_1, w_2\in W$ satisfy $gw_1=w_2$ for some $g\in T$ then $w_1$ and $w_2$ are connected by a path of length at most $C$ in which each edge is an $S_k$-edge with both endpoints in $\cV_k$ for some $1\le k\le n$.
\end{enumerate}

Let $T\in \overline{\cF}(G)$. Then we have $C, n$, $S_k$ for $1\le k\le n$, and $W$ and $\cV_k$ for $1\le k\le n$ as above.
We now verify conditions (i)-(iii) in Definition~\ref{D-CSC 1} as referenced in Definition~\ref{D-CSC 3}.

Let $g\in T$. For each $x\in W\cap g^{-1}W$, we can find $g_1, \dots, g_l\in G$ for some $1\le l\le C$ such that $g=g_lg_{l-1}\cdots g_1$ and for each $1\le j\le l$ one has $g_j\in S_{k_j}$ and $g_{j-1}\cdots g_1x, g_jg_{j-1}\cdots g_1x\in \cV_{k_j}$ for some $1\le k_j\le n$.
Then we can find a finite Borel partition $\sC_g$ of $W\cap g^{-1}W$ such that
\[
|\sC_g|\le Cn^C\Big(\max_{1\le k\le n}|S_k|\Big)^C
\]
and for each $A\in \sC_g$ we can choose the same $l, g_1, \dots, g_l, k_1, \dots, k_l$ for all $x\in A$.
Then for all $1\le j\le l$ the sets $g_{j-1}\cdots g_1A$ and $g_jg_{j-1}\cdots g_1A$ are contained in $\cV_{k_j}$.

Denote by $\sC^\sharp$ the finite partition of $X$ generated by $W, \cV_1, \dots, \cV_n$ and $\sC_g$ for $g\in T$.
Set $F^\sharp=(T\cup S\cup \bigcup_{k=1}^nS_k)^{100C}\in \cF(G)$. Set $D=|T|C^2n^C(\max_{1\le k\le n}|S_k|)^C>0$, and take $\delta>0$ with $3\delta |T|\Upsilon(T)\le 1/4$.
Take
\[
0<\delta^\sharp\le \min \bigg\{ \bigg(4\sum_{k=1}^n\Upsilon(S_k)\bigg)^{-1} , \delta/(|S|(|T|+D+2)) \bigg\} .
\]

Let $\pi: G\rightarrow \Sym(V)$ be a sofic approximation for $G$ with
$\Hom_\mu(\sC^\sharp, F^\sharp, \delta^\sharp, \pi)\neq\emptyset$ which is good enough so that
$\m(V_{F^\sharp})>1-\delta^\sharp$, where $V_{F^\sharp}$ denotes the set of all $v\in V$ satisfying $\pi_{gh}v=\pi_g\pi_hv$ for all $g, h\in F^\sharp$ and $\pi_gv\neq \pi_h v$ for all distinct $g, h\in F^\sharp$. Take $\varphi\in \Hom_\mu(\sC^\sharp, F^\sharp, \delta^\sharp, \pi)$. Then $\varphi$ is an algebra homomorphism $\alg(\sC^\sharp_{F^\sharp})\rightarrow \Pb_V$ satisfying
\begin{enumerate}
\item[(i)] $\sum_{A\in \sC^\sharp}\m(\pi_g\varphi(A)\Delta \varphi(gA))\le \delta^\sharp$ for all $g\in F^\sharp$, and
\item[(ii)] $\sum_{A\in \sC^\sharp_{F^\sharp}}|\m(\varphi(A))-\mu(A)|\le \delta^\sharp.$
\end{enumerate}

Let $g\in T$ and $A\in \sC_g$. Then we have $l, g_1, \dots, g_l, k_1, \dots, k_l$ as above. Denote by $W_{g, A}''$ the set $\bigcup_{1\le j\le l}(\pi_{g_jg_{j-1}\cdots g_1}^{-1}(\varphi(g_jg_{j-1}\cdots g_1A))\Delta \varphi(A))$. Then
\begin{align*}
\m(W_{g, A}'')&\le \sum_{j=1}^l\m(\pi_{g_jg_{j-1}\cdots g_1}^{-1}(\varphi(g_jg_{j-1}\cdots g_1A))\Delta \varphi(A))\\
&=\sum_{j=1}^l\m(\varphi(g_jg_{j-1}\cdots g_1A)\Delta \pi_{g_jg_{j-1}\cdots g_1}\varphi(A))\le C\delta^\sharp.
\end{align*}
Also set $W_g''=\bigcup_{A\in \sC_g}W_{g, A}''$ and $W''=\bigcup_{g\in T}W_g''$. Then
\[
\m(W_g'')\le |\sC_g|C\delta^\sharp\le C^2n^C\Big(\max_{1\le k\le n}|S_k|\Big)^C\delta^\sharp
\]
and
\[
\m(W'')\le |T|C^2n^C\Big(\max_{1\le k\le n}|S_k|\Big)^C\delta^\sharp=D\delta^\sharp.
\]

Set $W'=\bigcup_{g\in T}(\pi_g^{-1}(\varphi(W)\cap V_{F^\sharp})\setminus \varphi(g^{-1}W))$. Note that
\[
W'=\bigcup_{g\in T}(\pi_{g^{-1}}(\varphi(W)\cap V_{F^\sharp})\setminus \varphi(g^{-1}W))\subseteq \bigcup_{g\in T}(\pi_{g^{-1}}\varphi(W)\setminus \varphi(g^{-1}W)),
\]
and hence
\[
\m(W')\le \sum_{g\in T}\m(\pi_{g^{-1}}\varphi(W)\Delta \varphi(g^{-1}W))\le |T|\delta^\sharp.
\]

Set $W^*=(\varphi(W)\cap V_{F^\sharp})\setminus (W'\cup W'')$, and $W^\dag=W^*\cup \pi_{e_G}^{-1}(V\setminus \bigcup_{g\in S}\pi_gW^*)$.
Then $\bigcup_{g\in S}\pi_gW^\dag=V$, verifying condition (ii) in Definition~\ref{D-CSC 1}.

We have
\begin{align*}
\m\bigg(\bigcup_{g\in S}\pi_gW^*\bigg)&\ge \m\bigg(\bigcup_{g\in S}\pi_g\varphi(W)\bigg)-|S|\m(W'\cup W''\cup(V\setminus V_{F^\sharp}))\\
&\ge \m\bigg(\varphi\bigg(\bigcup_{g\in S}gW\bigg)\bigg)-\m\bigg(\bigg(\bigcup_{g\in S}\pi_g\varphi(W)\bigg)\Delta \varphi\bigg(\bigcup_{g\in S}gW\bigg)\bigg) \\
&\hspace*{25mm} \ -|S|(|T|\delta^\sharp+D\delta^\sharp+\delta^\sharp)\\
&= 1-\m\bigg(\bigg(\bigcup_{g\in S}\pi_g\varphi(W)\bigg)\Delta \bigcup_{g\in S}\varphi(gW)\bigg)-|S|(|T|+D+1)\delta^\sharp\\
&\ge 1-\sum_{g\in S}\m(\pi_g\varphi(W)\Delta \varphi(gW))-|S|(|T|+D+1)\delta^\sharp\\
&\ge 1-|S|\delta^\sharp-|S|(|T|+D+1)\delta^\sharp\ge 1-\delta,
\end{align*}
and hence
\[
\m(W^\dag\setminus W^*)\le \m\bigg(V\setminus \bigcup_{g\in S}\pi_gW^*\bigg)\le \delta.
\]
Put $\cV_k^\dag=\varphi(\cV_k)$ for $1\le k\le n$, $S_{n+1}=T\in \overline{\cF}(G)$, and
\[
\cV_{n+1}^\dag=\bigcup_{g\in T}((W^\dag\setminus W^*)\cup \pi_g(W^\dag\setminus W^*)\cup \pi_g^{-1}(W^\dag\setminus W^*)).
\]
Then
\[
\m(\cV_{n+1}^\dag)\le (2|T|+1)\m(W^\dag\setminus W^*)\le 3\delta |T|,
\]
and hence
\begin{align*}
\sum_{k=1}^{n+1}\Upsilon(S_k)\m(\cV_k^\dag)\le 3\delta |T|\Upsilon(T)+\sum_{k=1}^n\Upsilon(S_k)(\mu(\cV_k)+\delta^\sharp)\le \frac14 + \frac12 +\delta^\sharp\sum_{k=1}^n\Upsilon(S_k)\le 1,
\end{align*}
verifying condition (i) in Definition~\ref{D-CSC 1}.

Let $g\in T$ and $w_1, w_2\in W^\dag$ with $\pi_gw_1=w_2$.
If $w_1\not\in W^*$ or $w_2\not\in W^*$, then $(w_1, w_2)$ is an $S_{n+1}$-edge with both endpoints in $\cV_{n+1}^\dag$.
Thus we may assume that $w_1, w_2\in W^*$.
Then $w_1=\pi_g^{-1}w_2\in \pi_g^{-1}(\varphi(W)\cap V_{F^\sharp})$. Since $w_1\notin W'$, we get $w_1\in \varphi(g^{-1}W)$. Thus
\[
w_1\in \varphi(W)\cap \varphi(g^{-1}W)=\varphi(W\cap g^{-1}W)=\varphi\bigg(\bigcup_{A\in \sC_g}A\bigg)=\bigcup_{A\in \sC_g}\varphi(A).
\]
We have $w_1\in \varphi(A)$ for some $A\in \sC_g$. Let $l, g_1, \dots, g_l, k_1, \dots, k_l$ be as above for this $A$.
Then for all $1\le j\le l$ the sets $g_{j-1}\dots g_1A$ and $g_jg_{j-1}\dots g_1A$ are contained in $\cV_{k_j}$.
Since $w_1\notin W_{g, A}''$, we have $\pi_{g_jg_{j-1}\dots g_1}w_1\in \varphi(g_jg_{j-1}\dots g_1A)$ for all $1\le j\le l$.
Thus $\pi_{g_{j-1}\dots g_1}w_1, \pi_{g_jg_{j-1}\dots g_1}w_1\in \varphi(\cV_{k_j})=\cV_{k_j}^\dag$ for all $1\le j\le l$.
Therefore $w_1$ and $w_2$ are connected by a path of length $l$ in which each edge is an $S_k$-edge with both endpoints in $\cV_k^\dag$ for some $1\le k\le n$, verifying condition (iii) in Definition~\ref{D-CSC 1}.
\end{proof}

\begin{theorem} \label{T-compare SC}
Consider the following conditions for an infinite $G$:
\begin{enumerate}
\item $G$ has property SC,
\item every free p.m.p.\ action $G\curvearrowright (X, \mu)$ has property SC,
\item there exists a nontrivial Bernoulli action of $G$ with property SC,
\item there exists a nontrivial Bernoulli action of $G$ with property sofic SC,
\item $G$ has property sofic SC,
\item $G$ is neither locally finite nor finitely generated and virtually free.
\end{enumerate}
We have (i)$\Leftrightarrow$(ii)$\Leftrightarrow$(iii)$\Rightarrow$(iv)$\Leftrightarrow$(v)$\Rightarrow$(vi).
Moreover, when $G$ is amenable all of these conditions are equivalent.
\end{theorem}

\begin{proof}
The equivalence of (i), (ii), and (iii) is the content of Proposition~3.5 of \cite{KerLi19}.
For (iii)$\Rightarrow$(iv) apply Proposition~\ref{P-SC to sofic SC}.
The implication (iv)$\Rightarrow$(v) follows from the fact that nontrivial Bernoulli actions have positive sofic entropy with respect to every sofic approximation sequence \cite{Bow10,KerLi11,Ker13},
while (v)$\Rightarrow$(iv) follows from the definitions 
and (v)$\Rightarrow$(vi) from Propositions~\ref{P-locally finite group}
and \ref{P-virtually free group}.

In the case that $G$ is amenable, Proposition~3.28 of \cite{KerLi19} asserts that (vi)$\Leftrightarrow$(i),
which gives us the equivalence of all of the conditions.
\end{proof}

A subgroup $G_0$ of $G$ is said to be {\it w-normal} in $G$ if there are a countable ordinal
$\gamma$ and a subgroup $G_\lambda$ of $G$ for each ordinal $0\le \lambda\le \gamma$
such that
\begin{enumerate}
\item for any $\lambda<\lambda'\le \gamma$ one has $G_\lambda\subseteq G_{\lambda'}$,
\item $G=G_\gamma$,
\item for each $\lambda<\gamma$ the group $G_\lambda$ is normal in $G_{\lambda+1}$,
\item for each limit ordinal $\lambda'\le \gamma$ one has $G_{\lambda'}=\bigcup_{\lambda<\lambda'}G_\lambda$.
\end{enumerate}
In conjunction with Theorem~\ref{T-compare SC} above, Theorem~3.29 of \cite{KerLi19} yields the following.

\begin{theorem}\label{T-w-normal amenable}
Suppose that $G$ has a w-normal subgroup $G_0$ which is amenable but neither locally finite nor virtually cyclic.
Then $G$ has property sofic SC.
\end{theorem}

\subsection{W-normal subgroups and property sofic SC}\label{SS-normal}

The proof of the following lemma applies some of the ideas from Section~8.1 of \cite{Aus16}
to the sofic framework.

\begin{lemma}\label{L-tree sofic}
Suppose that $G$ is finitely generated and not virtually cyclic, and
let $A$ be a generating set for $G$ in $\overline{\cF}(G)$.
Then there is a constant $b>0$ such that
given any
\begin{enumerate}
\item group $H$ containing $G$ as a subgroup,

\item finite subset $F$ of $H$, and

\item $r, M\in\Nb$ and $\delta>0$
\end{enumerate}
one can find, for any
good enough sofic approximation $\pi: H\rightarrow \Sym(V)$ for $H$,
sets $Z\subseteq\cV\subseteq V_F$, where $V_F$ denotes the set of all $v\in V$ satisfying $\pi_{gh}=\pi_g\pi_hv$ for all $g, h\in F$ and $\pi_gv\neq \pi_hv$ for all distinct $g, h\in F$, such that $\big|\bigcup_{g\in A^{2r}}\pi_g\cV\big|/|V|\ge 1-\delta$, $|\cV|\le b|V|/r$, $|Z|\le |V|/M$, and every point of $\cV$ is connected to some point of $Z$ by an $A$-path of length at most $2M$ with all vertices in $\cV$.
\end{lemma}

\begin{proof}
Since $G$ is not virtually cyclic, there exists a $c>0$ such that $|A^n| \geq cn^2$ for all $n\in\Nb$
(Corollary~3.5 of \cite{Man12}). Set $b=5/c$.
Let $H, F, r, M, \delta$, and $\pi$ be as in the lemma statement. Set $N=|A|^{3M}$. Take $k\in \Nb$ such that $|A^{kr}|\ge N|A^{2r}|$.

Denote by $V'$ the set of all $v\in V$ satisfying $\pi_{gh}v=\pi_g\pi_hv$ for all $g, h\in (F\cup A)^{100(M+r)}$ and $\pi_gv\neq \pi_hv$ for all distinct $g, h\in (F\cup A)^{100(M+r)}$.
Denote by $V''$ the set of all $v\in V$ satisfying $\pi_{gh}v=\pi_g\pi_hv$ for all $g, h\in (F\cup A)^{(200+k)(M+r)}$ and $\pi_gv\neq \pi_hv$ for all distinct $g, h\in (F\cup A)^{(200+k)(M+r)}$.
Then $\pi_gV''\subseteq V'$ for every $g\in A^{(k+6)r}$.
Assuming that $\pi$ is a good enough sofic approximation, we have $|V''|/|V|\ge 1-\delta$.
Take a maximal $(A,r)$-separated subset $W$ of $V''$, and also take a maximal $(A, r)$-separated subset $W'$ of $V'$ containing $W$. Then we have $\bigcup_{g\in A^{2r}}\pi_gW\supseteq V''$, and
hence
\[
\frac{1}{|V|} \bigg|\bigcup_{g\in A^{2r}}\pi_gW\bigg| \ge \frac{|V''|}{|V|} \ge 1-\delta.
\]

Let $w\in W$.  Set $T_w=W'\cap \pi_{A^{(k+2)r}}w$. Note that $\pi_{A^{kr}}w\subseteq V'\subseteq \pi_{A^{2r}}W'$. For each $g\in A^{kr}$ we have $\pi_gw\in \pi_{A^{2r}}z$  for some $z\in W'$. Then $z\in \pi_{A^{2r}}\pi_gw\subseteq \pi_{A^{(k+2)r}}w$, and hence $z\in T_w$. Thus 
\[
\pi_{A^{kr}}w\subseteq \pi_{A^{2r}}T_w.
\]
Therefore
\[
|T_w|\ge \frac{|A^{kr}|}{|A^{2r}|}\ge N.
\]

Set $T=\bigcup_{w\in W}T_w\subseteq W'$.
We have
\begin{align*}
|A^r| |W'|
= \bigg| \bigsqcup_{w\in W'} \pi_{A^r} w \bigg|
\leq |V|
\end{align*}
whence
\begin{align}\label{E-E2}
|T|\le |W'| \leq \frac{|V|}{|A^r|} \leq \frac{|V|}{cr^2} .
\end{align}

Let $w\in W$. Let $(T_w, E_w)$ be the graph whose edges are
those pairs of vertices which can be joined by an $A$-path of length at most $4r+1$, and let us
show that it is connected.
It is enough to demonstrate that a given $v\in T_w$ is connected to $w$ by a path in $(T_w, E_w)$.
Choose a shortest $A$-path from $w$ to $v$. For each vertex $z$ in this path contained in $\pi_{A^{kr}}w$, the fact that $z\in \pi_{A^{kr}}w\subseteq \pi_{A^{2r}}T_w$ means that
we can connect $z$ to some $u_z\in T_w$ by an $A$-path $p_z$ of length at most $2r$.
By inserting $p_z$ and its reverse at $z$, we construct an $A$-path
from $w$ to $v$ in which points of $T_w$ appear in every interval of length $4r+1$.
Therefore $v$ is connected to $w$ by some path in $(T_w, E_w)$, showing that $(T_w, E_w)$ is connected.

Consider the graph $(T , E )$ whose edges are
those pairs of vertices which can be joined by an $A$-path of length at most $4r+1$.
From the above, every connected component of this graph has at least $N$ points.
Starting with $(T , E )$,
we recursively build a sequence of graphs with vertex set $T$
by removing one edge at each stage so as to destroy some cycle at that stage,
until there are no more cycles left and we arrive at a subgraph $(T, E')$ such that $(T, E)$ and $(T, E')$
have the same connected components and each connected component of $(T, E')$ is a tree.

For each pair $(v,w)$ in $E'$, we choose an $A$-path in $\pi_{A^{4r}}T$ joining $v$ to $w$
of length at most $4r+1$.
Denote by $\cV$ the collection
of all vertices which appear in one of these paths.
Then $\cV\subseteq \pi_{A^{4r}}T\subseteq V'\subseteq V_F$.
Note that each $A$-connected component of $\cV$ has at least $N$ points, and $W\subseteq T\subseteq \cV$. Thus
\[
\frac{1}{|V|} \bigg|\bigcup_{g\in A^{2r}}\pi_g\cV\bigg|
\ge \frac{1}{|V|} \bigg|\bigcup_{g\in A^{2r}}\pi_gW\bigg|
\ge 1-\delta.
\]
Moreover, using \eqref{E-E2} we have
\begin{align*}
|\cV| \leq |T|+4r|E'| \leq (4r+1)|T|
\leq 5r\cdot \frac{|V|}{cr^2} =\frac{b|V|}{r} .
\end{align*}

Let $\cC$ be an $A$-connected component of $\cV$. Denote by $(\cC, E_{\cC})$ the graph whose edges are the pairs
$(w, v)\in \cC^2$ such that $\pi_gw=v$ for some $g\in A$. Then $(\cC, E_{\cC})$ is connected.
Endow $\cC$ with the geodesic distance $\rho$ induced from $E_{\cC}$.
Take a maximal subset $Z_{\cC}$ of $\cC$ which is $(\rho, M)$-separated
in the sense that the $M$-balls $\{ v\in \cC : \rho (v,z) \leq M \}$ for $z\in Z_{\cC}$
are pairwise disjoint.
Then $Z_{\cC}$ is $(\rho, 2M)$-spanning in $\cC$, i.e., every point of $\cC$ is connected to some point of $Z_\cC$ by an $A$-path of length at most $2M$ with all vertices in $\cC$. Since $|\cC|\ge N=|A|^{3M}>|A|^{2M}$,  we have $|Z_{\cC}|\ge 2$. Then $|\cC\cap \pi_{A^{M}}z|\ge M$ for every $z\in Z_{\cC}$.
Since the sets $\cC\cap \pi_{A^M}z$ for $z\in Z_{\cC}$ are pairwise disjoint, we get
\[
|Z_{\cC}|M\le \sum_{z\in Z_{\cC}}|\cC\cap \pi_{A^M}z|\le |\cC|.
\]

Denote by $Z$ the union of the sets $Z_{\cC}$ where $\cC$ runs over all $A$-connected components of $\cV$. Then every point of $\cV$ is connected to some point of $Z$ by an $A$-path of length at most $2M$ with all vertices in $\cV$, and $|Z|/|V|\le 1/M$.
\end{proof}

For the definition of w-normality, see the paragraph before Theorem~\ref{T-w-normal amenable}.

\begin{proposition} \label{P-normal sofic SC}
Suppose that $G$ has a w-normal subgroup $G^\flat$ with property sofic SC.
Then $G$ has property sofic SC.
\end{proposition}

\begin{proof}
Suppose first that $G^\flat$ is locally virtually cyclic. Then $G^\flat$ is amenable and,
by Lemma~\ref{L-finite group} and Theorem~\ref{T-compare SC}, neither locally finite nor virtually cyclic.
It follows by Theorem~\ref{T-w-normal amenable} that $G$ has property sofic SC, as desired.

Suppose now that $G^\flat$ is not locally virtually cyclic. In this case we will first carry out the argument
under the assumption that $G^\flat$ is normal in $G$.
Take a finitely generated subgroup $G_0$ of $G^\flat$ such that $G_0$ is not virtually cyclic.
Take an $S_1\in \overline{\cF}(G_0)$ generating $G_0$.
Let $b > 0$ be as given by Lemma~\ref{L-tree sofic}
for the group $G_0$ and generating set $S_1$.

Let $\Upsilon$ be a function $\cF(G)\rightarrow [0, \infty)$.
Choose an $r\in\Nb$ large enough so that
\begin{align}\label{E-r}
3b\Upsilon(S_1)\leq r.
\end{align}
Set $S=S_1^{2r}\in \overline{\cF}(G_0)$.

Consider the restriction of $3\Upsilon$ to $\cF(G^\flat)$. Since $G^\flat$ has property sofic SC,
there exists an $S^\flat\in \overline{\cF}(G^\flat)$ such that for any $T^\flat\in \overline{\cF}(G^\flat)$ there are $C^\flat , n^\flat\in \Nb$ and $S^\flat_1, \dots, S^\flat_{n^\flat}\in \overline{\cF}(G^\flat)$ such that for any good enough sofic approximation $\pi: G\rightarrow \Sym(V)$ for $G$ there are subsets $W^\flat$ and $\cV^\flat_k$ of $V$ for $1\le k\le n^\flat$   satisfying the following conditions:
\begin{enumerate}
\item $\sum_{k=1}^{n^\flat} 3\Upsilon(S^\flat_{k})\m(\cV^\flat_{k})\le 1$,
\item $\bigcup_{g\in S^\flat}\pi_gW^\flat=V$,
\item if $w_1, w_2\in W^\flat$ satisfy $\pi_gw_1=w_2$ for some $g\in T^\flat$ then $w_1$ and $w_2$ are connected by a path of length at most $C^\flat$ in which each edge is an $S^\flat_k$-edge with both endpoints in $\cV^\flat_k$ for some $1\le k\le n^\flat$.
\end{enumerate}

Let $T\in \overline{\cF}(G)$.
Set
\[
S_2=S^\flat T S^\flat \in \overline{\cF}(G).
\]
Take an $M\in \Nb$ large enough so that
\begin{align}\label{E-M large1}
M\ge 12\Upsilon(S_2)|S_2|.
\end{align}

Set $T^\flat =\bigcup_{g\in T}(S^\flat S_1^{2M}gS_1^{2M}g^{-1}S^\flat\cup S^\flat gS_1^{2M}g^{-1}S_1^{2M}S^\flat)\in \overline{\cF}(G^\flat)$. Then we have $C^\flat$, $n^\flat$, and $S^\flat_{k}$ for $1\le k\le n^\flat$ as above. Set
\[
C=4M+2+C^\flat\in \Nb,
\]
and $F=(S_1\cup T\cup S^\flat\cup \bigcup_{k=1}^{n^\flat}S_k^\flat)^{100MCr}\in \cF(G)$.

Now let $\pi: G\rightarrow \Sym(V)$ be a good enough sofic approximation for $G$.
By Lemma~\ref{L-tree sofic} we can find sets $Z\subseteq\cV_1\subseteq V_F$, where $V_F$ denotes the set of $v\in V$ satisfying $\pi_{gh}=\pi_g\pi_hv$ for all $g, h\in F$ and $\pi_gv\neq \pi_hv$ for all distinct $g, h\in F$, such that $\m\big(\bigcup_{g\in S}\pi_g\cV_1\big)\ge 1-1/M$,  $\m(\cV_1)\le b/r$, $\m(Z)\le 1/M$, and every point of $\cV_1$ is connected to some point of $Z$ by an $S_1$-path of length at most $2M$ with all vertices in $\cV_1$.
Note that
\[
\Upsilon(S_1)\m(\cV_1)\le \Upsilon(S_1)\frac{b}{r}\overset{\eqref{E-r}}{\le} \frac{1}{3}.
\]
Set $W=\cV_1\cup \pi_{e_G}^{-1}(V\setminus \bigcup_{g\in S}\pi_g \cV_1)\subseteq V$. Then $\bigcup_{g\in S}\pi_g W=V$, which verifies condition (ii) in Definition~\ref{D-CSC 1}.

Set $\cV_2=\big(\bigcup_{g\in T}((W\setminus \cV_1)\cup \pi_g (W\setminus \cV_1)\cup \pi_g^{-1}(W\setminus \cV_1))\big)\cup \bigcup_{g\in S_2}\pi_g Z\subseteq V$.
Then
\begin{align*}
\m(\cV_2)\le 3|T|\m(W\setminus \cV_1)+|S_2|\m(Z)\le 3|S_2|\m(V\setminus \bigcup_{g\in S}\pi_g \cV_1)+|S_2|/M\le 4|S_2|/M,
\end{align*}
and hence
\[
\Upsilon(S_2)\m(\cV_2)\le 4\Upsilon(S_2) |S_2|/M\overset{\eqref{E-M large1}}{\le}
\frac{1}{3}.
\]
Assuming that $\pi$ is a good enough sofic approximation for $G$, we have
$W^\flat$ and $\cV^\flat_{k}$ for $1\le k\le n^\flat$ as above, in which case
\[
\sum_{k=1}^{n^\flat}\Upsilon(S^\flat_{k})\m(\cV^\flat_{k})\le \frac{1}{3}.
\]
Putting the above estimates together we get
\[
\Upsilon(S_1)\m(\cV_1)+\Upsilon(S_2)\m(\cV_2)+\sum_{k=1}^{n^\flat}\Upsilon(S^\flat_{k})\m(\cV^\flat_{k})\le 1,
\]
which verifies condition (i) in Definition~\ref{D-CSC 1}.

Let $g\in T$ and $w_1, w_2\in W$ be such that $\pi_gw_1=w_2$. If either $w_1\in W\setminus \cV_1$ or $w_2\in W\setminus \cV_1$, then $(w_1, w_2)$ is an $S_2$-edge with both endpoints in $\cV_2$. Thus we may assume that $w_1, w_2\in \cV_1$.
For $i=1, 2$, we can connect
$w_i$  to some $z_i\in Z$  by an $S_1$-path of length at most $2M$ with all vertices in $\cV_1$.
Then $w_i=\pi_{t_i}z_i$ for some $t_i\in S_1^{2M}$.
We have $\pi_gz_1=\pi_{a_1} u_1$ for some $u_1\in W^\flat$ and $a_1\in S^\flat$, and $z_2= \pi_{a_2} u_2$ for some $u_2\in W^\flat$ and $a_2\in S^\flat$.
Note that $a_1^{-1}g$ and $a_2^{-1}$ are both in $S_2$.
Since $\pi_{a_1^{-1}g}z_1=u_1$, the pair $(z_1, u_1)$ is an $S_2$-edge with both endpoints in $\cV_2$. Also, since $\pi_{a_2^{-1}}z_2=u_2$ the pair $(w_2, u_2)$ is an $S_2$-edge with both endpoints in  $\cV_2$.
Note that
\begin{align*}
\pi_{a_2^{-1}t_2^{-1}gt_1g^{-1}a_1} u_1&=\pi_{a_2^{-1}}\pi_{t_2^{-1}}\pi_{g}\pi_{t_1}\pi_{g^{-1}}\pi_{a_1} u_1\\
&=\pi_{a_2^{-1}}\pi_{t_2^{-1}}\pi_g\pi_{t_1}z_1\\
&=\pi_{a_2^{-1}}\pi_{t_2^{-1}}\pi_g w_1\\
&=\pi_{a_2^{-1}}\pi_{t_2^{-1}}w_2\\
&=\pi_{a_2^{-1}}z_2\\
&=u_2.
\end{align*}
Since $a_2^{-1}t_2^{-1}gt_1g^{-1}a_1\in S^\flat S_1^{2M}gS_1^{2M}g^{-1}S^\flat\subseteq T^\flat$, this means that $u_2\in \pi_{T^\flat}u_1$. Then $u_1$ and $u_2$ are connected by a path of length at most $C^\flat$ in which each edge is an $S^\flat_k$-edge with both endpoints in $\cV^\flat_k$ for some $1\le k\le n^\flat$.
Therefore $w_1$ and $w_2$ are connected by a path of length at most $4M+2+C^\flat=C$ in which each edge is either an $S_j$-edge with both endpoints in $\cV_j$ for some $1\le j\le 2$ or an $S^\flat_k$-edge with both endpoints in $\cV^\flat_k$ for some $1\le k\le n^\flat$,
verifying condition (iii) in Definition~\ref{D-CSC 1}.

Notice that the set $S$ used in the above verification of property sofic SC for $G$
is contained in $G_0$, which can be any non-virtually-cyclic finitely generated subgroup of $G^\flat$,
and only depends on the restriction of $\Upsilon$ to $\overline{\cF}(G_0)$.
This has the consequence that if $G_1 , G_2 , \dots$ is a sequence of countable groups
such that $G_n$ is a normal subgroup of $G_{n+1}$ for each $n$
and $G_1$ is not locally virtually cyclic and has property sofic SC
then the group $\bigcup_{n=1}^\infty G_n$ has property sofic SC.
Indeed we can fix a finitely generated subgroup $G_1'$ of $G_1$ which is not virtually cyclic
and apply the above argument recursively taking $G_0 = G_1'$, $G^\flat = G_n$, and $G=G_{n+1}$ at the $n$th stage
to deduce that $G_{n+1}$ has property sofic SC, and if the function $\Upsilon$ is taken at each stage
to be the restriction of a prescribed function
$\cF (\bigcup_{n=1}^\infty G_n ) \to [0,\infty )$
then we can use the same set $S$ for all $n$, showing that $\bigcup_{n=1}^\infty G_n$
has property sofic SC.
It follows by ordinal well-ordering that if $G^\flat$ is merely assumed to be w-normal in $G$ then
we can still conclude that $G$ has property sofic SC.
\end{proof}

\subsection{Product groups}\label{SS-products}

Let $G$ and $H$ be countable groups. Let $\pi : G\to\Sym (V)$ and $\sigma : H\to\Sym (W)$
be sofic approximations. The {\it product sofic approximation} $\pi\times\sigma : G\times H \to \Sym (V\times W)$
is defined by
\[
(\pi\times\sigma )_{(g,h)} (v,w) = (\pi_g (v) , \sigma_h (w))
\]
for all $g\in G$, $h\in H$, $v\in V$, and $w\in W$.
Note that if $\{ \pi_k \}$ and $\{ \sigma_k \}$ are sofic approximation sequences for $G$ and $H$,
respectively, then $\{ \pi_k\times\sigma_k \}$ is a sofic approximation sequence for $G\times H$.

\begin{proposition}\label{P-product}
Let $G$ and $H$ be countably infinite groups.
Let $\sS$ be the collection of product sofic approximations for $G\times H$.
Then $G\times H$ has property $\sS$-SC if and only if at least one of $G$ and $H$ is not locally finite.
\end{proposition}

\begin{proof}
If $G$ and $H$ are both locally finite then $G\times H$ is locally finite and hence does not have
property $\sS$-SC by Proposition~\ref{P-locally finite group}.
Suppose then that at least one of $G$ and $H$ is not locally finite.
Take two nontrivial Bernoulli actions $G\curvearrowright (X,\mu )$ and $H\curvearrowright (Y,\nu )$.
By Proposition~3.32 of \cite{KerLi19} the p.m.p.\ action
$G\times H\curvearrowright (X \times Y ,\mu \times\nu)$
given by $(g,h)(x,y) = (gx,hy)$ for all $g\in G$, $h\in H$, $x\in X$, and $y\in Y$
has property SC, and hence has property sofic SC by Proposition~\ref{P-SC to sofic SC}.

By \cite{Bow10,Ker13}, for every finite partition $\sC$ of $X$, $F\in\cF (G)$ containing $e_G$,
and $\delta > 0$ one has $\Hom_\mu (\sC , F,\delta ,\pi ) \neq \emptyset$
for every sufficiently good sofic approximation $\pi$ for $G$,
and for every finite partition $\sD$ of $Y$, $L\in\cF (H)$ containing $e_H$, and $\delta > 0$ one has $\Hom_\nu (\sD , L,\delta ,\sigma ) \neq \emptyset$
for every sufficiently good sofic approximation $\sigma$ for $H$.
Given such sofic approximations $\pi : G\to\Sym (V)$ and $\sigma : H\to\Sym (W)$
and $\varphi\in\Hom_\mu (\sC , F,\delta ,\pi )$ and $\psi\in\Hom_\nu (\sD , L,\delta ,\sigma )$
we have a homomorphism $\zeta : \alg (\sC_F \times \sD_L ) = \alg ((\sC\times\sD )_{F\times L} )\to \Pb_{V\times W}$
determined by $\zeta (C\times D) = \varphi (C)\times \psi (D)$ for $C\in\sC_F$ and $D\in\sD_L$,
and one can readily verify that $\zeta$ belongs to
$\Hom_{\mu\times\nu} (\sC\times\sD , F\times L,2\delta ,\pi\times\sigma )$,
showing that this set of homomorphisms is nonempty. Since the algebra of subsets of $X\times Y$ generated
by products of finite partitions is dense in the $\sigma$-algebra with respect to the pseudometric
$d(A,B) = (\mu\times\nu )(A\Delta B)$, it follows by a simple approximation argument that for every finite partition
$\sE$ of $X\times Y$, finite set $e_{G\times H} \in K\subseteq G\times H$, and $\delta > 0$
one has $\Hom_{\mu\times\nu} (\sE , K,\delta ,\pi\times\sigma )\neq\emptyset$
for all good enough sofic approximations $\pi : G\to\Sym (V)$ and $\sigma : H\to\Sym (W)$.
Since the action $G\times H\curvearrowright (X \times Y ,\mu \times\nu)$ has property
sofic SC, it follows that $G\times H$ has property $\sS$-SC.
\end{proof}

\subsection{Property sofic SC under continuous orbit equivalence}\label{SS-invariant cts}

\begin{proposition} \label{P-sofic SC under COE}
Let $G\curvearrowright X$ and $H\curvearrowright Y$ be topologically free continuous actions on compact metrizable spaces which are continuously orbit equivalent. Suppose that $G\curvearrowright X$ has property sofic SC. Then $H\curvearrowright Y$ has property sofic SC.
\end{proposition}

To prove this proposition
we may assume that $X=Y$ and that the identity map of $X$ provides a continuous orbit equivalence between the actions $G\curvearrowright X$ and $H\curvearrowright X$. Let $\kappa : G\times X\rightarrow H$ and
$\lambda : H\times X \rightarrow G$ be the associated cocycles.

The actions of $G$ and $H$ generate an action $G*H \curvearrowright X$ of their free product
via the canonical embeddings of $G$ and $H$ into $G*H$. Since the actions of $G$ and $H$ are topologically
free, we can find a $G$-invariant dense $G_\delta$ set $W_1 \subseteq X$ on which $G$ acts freely
and an $H$-invariant dense $G_\delta$ set $W_2 \subseteq X$ on which $H$ acts freely.
Set $X_0 = \bigcap_{s\in G*H} s(W_1 \cap W_2 )$. Then $X_0$ is a $G*H$-invariant
dense $G_\delta$ subset of $X$ on which both $G$ and $H$ act freely.

Fix a compatible metric $d$ on $X$ which gives $X$ diameter no bigger than $1$.
For each $g\in G$ there is an $\eta_g>0$ such that for any $x, y\in X$ with $d(x, y)\le \eta_g$ one has
$\kappa(g, x)=\kappa(g, y)$, and likewise
for each $s\in H$ there is an $\eta_s>0$ such that for any $x, y\in X$ with $d(x, y)\le \eta_s$
one has $\lambda(s, x)=\lambda(s, y)$. We put $\eta_F=\min_{g\in F}\eta_g>0$ for a nonempty finite set $F\subseteq G$,
and $\eta_L=\min_{s\in L}\eta_s>0$ for a nonempty finite set $L\subseteq H$.

We will need the following lemma, which will also be of use in the proof of Theorem~\ref{T-continuous OE}.

\begin{lemma} \label{L-approximation for group}
Let $L\in \overline{\cF}(H)$, and $0<\tau<1$. Set $F=\lambda(L^2, X)\in \overline{\cF}(G)$ and
\[
\tau'=\min\big\{\eta_{L^2}\tau^{1/2}/(8|F|)^{1/2}, \tau/(22|F|^2)\big\}>0 .
\]
Let $\pi: G\rightarrow \Sym(V)$ be an $(F, \tau')$-approximation for $G$. Let $\varphi\in \Map_d(F, \tau', \pi)$ be such that $\varphi(V)\subseteq X_0$. Define $\sigma': H\rightarrow V^V$ by
\[
\sigma'_tv=\pi_{\lambda(t, \varphi(v))}v
\]
for $t\in H$ and $v\in V$. Then there is an $(L, \tau)$-approximation $\sigma: H\rightarrow \Sym(V)$ for $H$ such that
$\rho_{\Hamm}(\sigma_t, \sigma'_t)\le \tau$ for all $t\in L^2$.
\end{lemma}

\begin{proof}
Denote by $V_F$ the set of all $v\in V$ satisfying $\pi_g\pi_hv=\pi_{gh}v$ for all $g, h\in F$ and $\pi_gv\neq \pi_hv$ for all distinct $g, h\in F$.
Then
\[
\m (V\setminus V_F)\le 2|F|^2\tau'\le \frac{\tau}{11}.
\]
Denote by $V_\varphi$ the set of all $v\in V$ satisfying
$d(\varphi(\pi_gv), g\varphi(v))\le \eta_{L^2}$ for all $g\in F$.
Then
\[
\m (V\setminus V_\varphi )\le |F|\bigg(\frac{\tau'}{\eta_{L^2}}\bigg)^2\le \frac{\tau}{8}.
\]

For all $s, t\in L^2$ and $v\in V_F\cap V_\varphi$,
since $\lambda(t, \varphi(v))\in F$ we have
\[
d(\varphi(\pi_{\lambda(t, \varphi(v))}v), \lambda(t, \varphi(v))\varphi(v))\le \eta_{L^2}
\]
and hence
\[
\lambda(s, \varphi(\pi_{\lambda(t, \varphi(v))}v))=\lambda(s, \lambda(t, \varphi(v))\varphi(v))=\lambda(s, t\varphi(v)),
\]
which yields
\begin{align*}
\sigma'_s\sigma'_tv=\pi_{\lambda(s, \varphi(\sigma'_tv))}\sigma'_tv
&=\pi_{\lambda(s, \varphi(\pi_{\lambda(t, \varphi(v))}v))}\pi_{\lambda(t, \varphi(v))}v\\
&=\pi_{\lambda(s, t\varphi(v))}\pi_{\lambda(t, \varphi(v))}v\\
&=\pi_{\lambda(s, t\varphi(v))\lambda(t, \varphi(v))}v\\
&=\pi_{\lambda(st, \varphi(v))}v\\
&=\sigma'_{st}v ,
\end{align*}
so that
\begin{align} \label{E-group}
\rho_{\Hamm}(\sigma'_s\sigma'_t, \sigma'_{st})\le
\m (V\setminus (V_F\cap V_\varphi))\le \frac{\tau}{11}+\frac{\tau}{8}=\frac{19\tau}{88}.
\end{align}

Note that $\sigma'_{e_H}=\pi_{e_G}$.
For each $t\in H$ choose a $\sigma_t\in \Sym(V)$ such that $\sigma_tv=\sigma'_tv$ for all $v\in V$ satisfying $\sigma'_{t^{-1}}\sigma'_tv=v$.
For each $t\in L^2$, taking $s=t^{-1}$ in \eqref{E-group} we conclude that
\begin{align*}
\rho_{\Hamm}(\sigma_t, \sigma'_t)&\le \rho_{\Hamm}(\sigma'_{t^{-1}}\sigma'_t, \id)\\
&\le \rho_{\Hamm}(\sigma'_{t^{-1}}\sigma'_t, \sigma'_{e_H})+\rho_{\Hamm}(\sigma'_{e_H}, \id)\\
&\le \frac{19\tau}{88}+\rho_{\Hamm}(\pi_{e_G}, \id)\\
&\le \frac{19\tau}{88}+\tau'\\
&\le \frac{19\tau}{88}+\frac{\tau}{22}=\frac{23\tau}{88},
\end{align*}
which in particular shows that
$\rho_{\Hamm}(\sigma_t, \sigma'_t)<\tau$.
For all $s,t \in L$ we then have
\begin{align*}
\rho_{\Hamm}(\sigma_s\sigma_t, \sigma_{st})&\le \rho_{\Hamm}(\sigma_s, \sigma'_s)+\rho_{\Hamm}(\sigma_t, \sigma'_t)+\rho_{\Hamm}(\sigma'_s\sigma'_t, \sigma'_{st})+\rho_{\Hamm}(\sigma_{st}, \sigma'_{st})\\
&\le \frac{23\tau}{88} + \frac{23\tau}{88}+\frac{19\tau}{88}+\frac{23\tau}{88}=\tau.
\end{align*}
For all distinct $s,t\in L^2$, since $\varphi(V)\subseteq X_0$ we have $\sigma'_sv\neq \sigma'_tv$ for all $v\in V_F$ and hence
\begin{align*}
\rho_{\Hamm}(\sigma_s, \sigma_t)&\ge \rho_{\Hamm}(\sigma'_s, \sigma'_t)-\rho_{\Hamm}(\sigma_s, \sigma'_s)-\rho_{\Hamm}(\sigma_t, \sigma'_t)\\
&\ge \m (V_F)-\frac{23\tau}{88}-\frac{23\tau}{88}\\
&\ge 1-\frac{\tau}{11}-\frac{23\tau}{44}> 1-\tau. \qedhere
\end{align*}
\end{proof}

\begin{proof}[Proof of Proposition~\ref{P-sofic SC under COE}]
Let $\Upsilon_H$ be a function $\cF(H)\rightarrow [0, \infty)$.
Define the function $\Upsilon_G: \cF(G)\rightarrow [0, \infty)$ by $\Upsilon_G(F)=2\Upsilon_H(\kappa(F, X))$.

Since the action $G\curvearrowright X$ has property sofic SC, there exists an $S_G\in \overline{\cF}(G)$  such that for any $T_G\in \overline{\cF}(G)$ there are $C_G, n_G\in \Nb$,  $S_{G, 1}, \dots, S_{G, n_G}\in \overline{\cF}(G)$, $L_G\in \overline{\cF}(G)$, and $0<\tau_G<1$  such that, for any $(L_G, \tau_G)$-approximation $\pi: G\rightarrow \Sym(V)$ for $G$ with $\Map_d(L_G, \tau_G, \pi)\neq \emptyset$, there are subsets $W_G$ and $\cV_{G, j}$ of $V$ for $1\le j\le n_G$  satisfying the following conditions:
\begin{enumerate}
\item $\sum_{j=1}^{n_G} \Upsilon_G(S_{G, j})\m(\cV_{G, j})\le 1$,
\item $\bigcup_{g\in S_G}\pi_gW_G=V$,
\item if $w_1, w_2\in W_G$ satisfy $\pi_gw_1=w_2$ for some $g\in T_G$ then $w_1$ and $w_2$ are connected by a path of length at most $C_G$ in which each edge is an $S_{G, j}$-edge with both endpoints in $\cV_{G, j}$ for some $1\le j\le n_G$.
\end{enumerate}

Set $S_H=\kappa(S_G, X)\in \overline{\cF}(H)$.

Let $T_H\in \overline{\cF}(H)$. Set $T_G=\lambda(T_H, X)\in \overline{\cF}(G)$.
Then we have $C_G$, $n_G$, $S_{G, j}$ for $1\le j\le n_G$, $L_G$, and $\tau_G$ as above. Set $C_H=C_G$, $n_H=n_G+1$, $S_{H, j}=\kappa(S_{G, j}, X)\in \overline{\cF}(H)$ for $1\le j\le n_G$, and $S_{H, n_H}=T_H\in \overline{\cF}(H)$. Take $0<\delta_H<1/(6\Upsilon_H(T_H)|T_H|)$. Also, set $A=L_G\cup S_G\cup T_G\cup \bigcup_{j=1}^{n_G}S_{G, j}\in \cF(G)$, $L_H=\kappa(A^{2(100+C_G)}, X)\in \cF(H)$, and
\begin{align*}
\tilde{\tau}_G &= \min\big\{(\tau_G/2)^2, \delta_H/(4|S_G|\cdot|A|^{2(100+C_G)})\big\}>0, \\
\tau_H &= \min\big\{\eta_{A^{2(100+C_G)}}\tilde{\tau}_G^{1/2}/(8|L_H|)^{1/2}, \tilde{\tau}_G/(22|L_H|^2), \tau_G/(2|L_H|^{1/2})\big\}>0.
\end{align*}

Let $\sigma: H\rightarrow \Sym(V)$ be an $(L_H, \tau_H)$-approximation for $H$ with $\Map_d(L_H, \tau_H/2, \sigma)\neq \emptyset$. Choose a $\varphi\in \Map_d(L_H, \tau_H/2, \sigma)$. Since $X_0$ is dense in $X$,
by perturbing $\varphi$ if necessary we may assume that $\varphi\in \Map_d(L_H, \tau_H, \sigma)$ and $\varphi(V)\subseteq X_0$. Define $\pi': G\rightarrow V^V$ by
\[
\pi'_gv=\sigma_{\kappa(g, \varphi(v))}v
\]
for all $v\in V$ and $g\in G$. By Lemma~\ref{L-approximation for group} there is an $(A^{100+C_G}, \tilde{\tau}_G)$-approximation $\pi: G\rightarrow \Sym(V)$ such that $\rho_{\Hamm}(\pi_g, \pi'_g)\le \tilde{\tau}_G$ for all $g\in A^{100+C_G}$. For each $g\in L_G\subseteq A^{100+C_G}$ we have
\begin{align*}
d_2(g\varphi, \varphi\pi_g)&\le d_2(g\varphi, \varphi\pi'_g)+d_2(\varphi\pi'_g, \varphi\pi_g)\\
&\le \bigg(\frac{1}{|V|}\sum_{v\in V}d(\kappa(g, \varphi(v))\varphi(v), \varphi(\sigma_{\kappa(g, \varphi(v))}v))^2\bigg)^{1/2}+\tilde{\tau}_G^{1/2}\\
&\le \bigg(\frac{1}{|V|}\sum_{v\in V}\sum_{t\in \kappa(g, X)}d(t\varphi(v), \varphi(\sigma_tv))^2\bigg)^{1/2}+\frac{\tau_G}{2}\\
&\le (\tau_H^2|L_H|)^{1/2}+\frac{\tau_G}{2} \\
&\le \tau_G.
\end{align*}
Thus $\varphi\in \Map_d(L_G, \tau_G, \pi)$. Then we have $W_G$ and $\cV_{G, j}$ for $1\le j\le n_G$ as above.

We now verify conditions (i)-(iii) in Definition~\ref{D-CSC 1} as referenced in Definition~\ref{D-CSC 2}.
Denote by $V_1$ the set of all $v\in V$ satisfying $\pi_{g_1g_2}v=\pi_{g_1}\pi_{g_2}v$ for all $g_1, g_2\in A^{100+C_G}$. Then $\m(V\setminus V_1)\le |A|^{2(100+C_G)}\tilde{\tau}_G$.
Also, denote by $V_2$ the set of $v\in V$ satisfying $\pi_g\pi_{g_1}v=\pi'_g\pi_{g_1}v$  for all $g\in A$ and $g_1\in A^{C_G}$. Then $\m(V\setminus V_2)\le \tilde{\tau}_G|A|^{1+C_G}$.
Set $W_H'=W_G\cap V_1\cap V_2$, and $W_H=W_H'\cup \sigma_{e_H}^{-1}(V\setminus \bigcup_{h\in S_H}\sigma_hW_H')$. Then $\bigcup_{h\in S_H}\sigma_hW_H=V$, verifying condition (ii) in Definition~\ref{D-CSC 1}.

Note that
\begin{align*}
\m(W_H\setminus W_H')&\le 1-\m\big(\bigcup_{h\in S_H}\sigma_hW_H'\big)\\
&\le 1-\m\big(\bigcup_{g\in S_G}\pi'_g W_H'\big)\\
&=1-\m\big(\bigcup_{g\in S_G}\pi_g W_H'\big)\\
&\le |S_G|(\m(V\setminus V_1)+\m(V\setminus V_2))\\
&\le |S_G|(\tilde{\tau}_G|A|^{2(100+C_G)}+ \tilde{\tau}_G|A|^{1+C_G})\\
&\le \delta_H.
\end{align*}
Put $\cV_{H, j}=\cV_{G, j}$ for all $1\le j\le n_G$, and
\[
\cV_{H, n_H}=\bigcup_{h\in T_H}((W_H\setminus W_H')\cup \sigma_h(W_H\setminus W_H')\cup \sigma_h^{-1}(W_H\setminus W_H')).
\]
Then
\[
\m (\cV_{H, n_H})\le (2|T_H|+1)\m (W_H\setminus W_H')\le 3|T_H|\delta_H,
\]
and hence
\begin{align*}
\sum_{j=1}^{n_H} \Upsilon_H(S_{H, j})\m(\cV_{H, j})&=\Upsilon_H(T_H)\m(\cV_{H, n_H})+\sum_{j=1}^{n_G} \Upsilon_H(\kappa(S_{G, j}, X))\m(\cV_{G, j})\\
&=\Upsilon_H(T_H)\m(\cV_{H, n_H})+\frac{1}{2}\sum_{j=1}^{n_G} \Upsilon_G(S_{G, j})\m(\cV_{G, j})\\
&\le 3\Upsilon_H(T_H)|T_H|\delta_H+\frac12 \le 1,
\end{align*}
verifying condition (i) in Definition~\ref{D-CSC 1}.
Let $h\in T_H$ and $w_1, w_2\in W_H$ with $\sigma_hw_1=w_2$.  If $w_1\not\in W_H'$ or $w_2\not\in W_H'$, then $(w_1, w_2)$ is an $S_{H, n_H}$-edge with both endpoints in $\cV_{H, n_H}$. We may thus assume that $w_1, w_2\in W_H'$. Then
\[
w_2=\sigma_hw_1=\pi'_{\lambda(h, \varphi(w_1))}w_1=\pi_{\lambda(h, \varphi(w_1))}w_1\in \pi_{T_G}w_1
\]
and so $w_1$ and $w_2$ are connected by a path of length at most $C_G$ in which each edge is an $S_{G, j}$-edge with both
endpoints in $\cV_{G, j}$ for some $1\le j\le n_G$.
It is easily checked that such an edge is also an $S_{H, j}$-edge. This verifies condition (iii) in Definition~\ref{D-CSC 1}.
\end{proof}

\subsection{Property sofic SC under bounded orbit equivalence}\label{SS-invariant bounded}

\begin{proposition} \label{P-sofic SC under BOE}
Let $G\curvearrowright (X, \mu)$ and $H\curvearrowright (Y, \nu)$ be free p.m.p.\ actions which are boundedly orbit equivalent. Suppose that $G\curvearrowright (X, \mu)$ has property sofic SC. Then so does $H\curvearrowright (Y, \nu)$.
\end{proposition}

To prove this proposition we may assume that $(X, \mu)=(Y, \nu)$, the actions $G\curvearrowright X$ and $H\curvearrowright X$ are free, and the identity map of $X$ provides a bounded orbit equivalence between the actions $G\curvearrowright X$ and $H\curvearrowright X$. Let $\kappa : G\times X\to H$ and $\lambda : H\times X\to G$ be the associated cocyles.

For each $g\in G$ denote by $\sP_g$ the finite Borel partition of $X$ consisting of the sets
$X_{g,t} := \{ x\in X : gx = tx \}$ for $t\in H$, and likewise for $t\in H$ denote by
$\sP_t$ the finite Borel partition of $X$ consisting of the sets $X_{g,t}$ for $g\in G$.
For every $F$ in $\cF(G)$ or $\cF(H)$, write ${}_F\sP=\bigvee_{g\in F}\sP_g$.

The following is a specialization of Lemma~4.2 in \cite{KerLi19} to the case of
bounded orbit equivalence, which permits a simplification of the statement.

\begin{lemma} \label{L-Shannon approximation for group measure}
Let $F\in \overline{\cF}(G)$ and set $L = \kappa (F^2 ,X) \in \overline{\cF}(H)$. Let $0<\tau<1$
and $0<\tau'\le \tau/(60|L|^2)$.
Let $\sigma : H\rightarrow \Sym(V)$ be an $(L, \tau')$-approximation for $H$. Let $\varphi\in \Hom_\mu({}_{F^2}\sP, L, \tau', \sigma)$. Let $\pi': F^2\rightarrow V^V$ be such that
\[
\pi'_gv=\sigma_{\kappa(g, A)}v
\]
for all $g\in F^2$, $A\in {}_{F^2} \sP$ and $v\in \varphi(A)$. Then there is an $(F, \tau)$-approximation $\pi : G\rightarrow \Sym(V)$ for $G$ such that
$\rho_{\Hamm}(\pi_g, \pi'_g)\le \tau/5$ for all $g\in F^2$.
\end{lemma}

\begin{proof}[Proof of Proposition~\ref{P-sofic SC under BOE}]
Let $\Upsilon_H$ be a function $\cF(H)\rightarrow [0, \infty)$.
Define a function $\Upsilon_G: \cF(G)\rightarrow [0, \infty)$ by $\Upsilon_G(F)=2\Upsilon_H(\kappa(F, X))$.

Since $G\curvearrowright (X, \mu)$ has property sofic SC, there exists an $S_G\in \overline{\cF}(G)$ such that for any $T_G\in \overline{\cF}(G)$ there are $C_G, n_G\in \Nb$, $S_{G, 1}, \dots, S_{G, n_G}\in \overline{\cF}(G)$,  a finite Borel partition $\sC_G$ of $X$, an $L_G\in \overline{\cF}(G)$, and $0<\tau_G<1$  such that, for any
$(L_G, \tau_G)$-approximation $\pi: G\rightarrow \Sym(V)$ for $G$ with $\Hom_\mu(\sC_G, L_G, \tau_G, \pi)\neq \emptyset$, there are subsets $W_G$ and $\cV_{G, j}$ of $V$ for $1\le j\le n_G$  satisfying the following conditions:
\begin{enumerate}
\item $\sum_{j=1}^{n_G} \Upsilon_G(S_{G, j})\m(\cV_{G, j})\le 1$,
\item $\bigcup_{g\in S_G}\pi_gW_G=V$,
\item if $w_1, w_2\in W_G$ satisfy $\pi_gw_1=w_2$ for some $g\in T_G$ then $w_1$ and $w_2$ are connected by a path of length at most $C_G$ in which each edge is an $S_{G, j}$-edge with both endpoints in $\cV_{G, j}$ for some $1\le j\le n_G$.
\end{enumerate}

Set $S_H=\kappa(S_G, X)\in \overline{\cF}(H)$.

Let $T_H\in \overline{\cF}(H)$. Set $T_G=\lambda(T_H, X)\in \overline{\cF}(G)$.
Then we have $C_G, n_G$, $S_{G, j}$ for $1\le j\le n_G$, $\sC_G$, $L_G$, and $\tau_G$ as above. Set $C_H=C_G$, $n_H=n_G+1$, $S_{H, j}=\kappa(S_{G, j}, X)\in \overline{\cF}(H)$ for $1\le j\le n_G$, and $S_{H, n_H}=T_H\in \overline{\cF}(H)$.
Take $0<\delta_H<1/(6\Upsilon_H(T_H)|T_H|)$.
Set $U=L_G\cup S_G\cup T_G\cup \bigcup_{j=1}^{n_G}S_{G, j}\in \cF(G)$,
$\sC_H=(\sC_G)_{L_G}\vee {}_{U^{2(100+C_G)}}\sP\vee {}_{T_H}\sP$, $L_H=\kappa(U^{2(100+C_G)}, X)\in \cF(H)$, and
\begin{align*}
\tilde{\tau}_G &= \min\big\{\tau_G/4, \delta_H/(2|S_G|\cdot|U|^{2(100+C_G)})\big\}>0,\\
\tau_H &= \min\big\{\tilde{\tau}_G/(60|L_H|^2), \tau_G/(2|\kappa(L_G, X)|)\big\}>0.
\end{align*}

Let $\sigma: H\rightarrow \Sym(V)$ be an $(L_H, \tau_H)$-approximation for $H$ with
$\Hom_\mu(\sC_H, L_H, \tau_H, \sigma)$ nonempty.
Take $\varphi\in \Hom_\mu(\sC_H, L_H, \tau_H, \sigma)$. Define $\pi': U^{2(100+C_G)}\rightarrow V^V$ by
\[
\pi'_gv=\sigma_{\kappa(g, A)}v
\]
for all $g\in U^{2(100+C_G)}$, $A\in \sP_g$, and $v\in \varphi(A)$. By Lemma~\ref{L-Shannon approximation for group measure}
there is a $(U^{100+C_G}, \tilde{\tau}_G)$-approximation $\pi: G\rightarrow \Sym(V)$ for $G$ such that
$\rho_{\Hamm}(\pi_g, \pi'_g)\le \tilde{\tau}_G$ for all $g\in U^{100+C_G}$.

Let $g\in L_G$. We have
\begin{align*}
\sum_{A\in \sC_G}\m(\pi_g\varphi(A)\Delta \varphi(gA))&\le \sum_{A\in \sC_G}\m(\pi_g\varphi(A)\Delta \pi_g'\varphi(A))+\sum_{A\in \sC_G}\m(\pi_g'\varphi(A)\Delta \varphi(gA))\\
&\le 2\rho_{\Hamm}(\pi_g, \pi_g')+\sum_{A\in \sC_G}\sum_{B\in \sP_g}\m(\pi_g'\varphi(A\cap B)\Delta \varphi(g(A\cap B))).
\end{align*}
For any $A\in \sC_G$ and $B\in \sP_g$, say $h=\kappa(g, B)\in L_H$, we have
$\pi_g'\varphi(A\cap B)=\sigma_h\varphi(A\cap B)$ and $\varphi(g(A\cap B))=\varphi(h(A\cap B))$, whence
\[
\sum_{A\in \sC_G}\m(\pi_g'\varphi(A\cap B)\Delta \varphi(g(A\cap B)))=\sum_{A\in \sC_G}\m(\sigma_h\varphi(A\cap B)\Delta \varphi(h(A\cap B)))\le \tau_H.
\]
Therefore
\begin{align*}
\sum_{A\in \sC_G}\m(\pi_g\varphi(A)\Delta \varphi(gA))&\le 2\rho_{\Hamm}(\pi_g, \pi_g')+\sum_{B\in \sP_g}\sum_{A\in \sC_G}\m(\pi_g'\varphi(A\cap B)\Delta \varphi(g(A\cap B)))\\
&\le 2\tilde{\tau}_G+\sum_{B\in \sP_g}\tau_H\\
&\le \frac{\tau_G}{2} + |\kappa(g, X)|\tau_H\le \tau_G.
\end{align*}
We also have
\[
\sum_{A\in (\sC_G)_{L_G}}|\m(\varphi(A))-\mu(A)|\le \sum_{B\in \sC_H}|\m(\varphi(B))-\mu(B)|\le \tau_H\le \tau_G.
\]
Therefore $\varphi\in \Hom_\mu(\sC_G, L_G, \tau_G, \pi)$. Then we have $W_G$ and $\cV_{G, j}$ for $1\le j\le n_G$ as above.

We now verify conditions (i)-(iii) in Definition~\ref{D-CSC 1} as referenced in Definition~\ref{D-CSC 3}.
Denote by $V_1$ the set of $v\in V$ satisfying $\pi_{g_1g_2}v=\pi_{g_1}\pi_{g_2}v$ for all $g_1, g_2\in U^{100+C_G}$. Then $\m(V\setminus V_1)\le \tilde{\tau}_G|U|^{2(100+C_G)}$.
Also denote by $V_2$ the set of $v\in V$ satisfying $\pi_g\pi_{g_1}v=\pi'_g\pi_{g_1}v$  for all $g\in U$ and $g_1\in U^{C_G}$. Then $\m(V\setminus V_2)\le \tilde{\tau}_G|U|^{1+C_G}$.
Set $W_H'=W_G\cap V_1\cap V_2$, and $W_H=W_H'\cup \sigma_{e_H}^{-1}(V\setminus \bigcup_{h\in S_H}\sigma_hW_H')$. Then $\bigcup_{h\in S_H}\sigma_hW_H=V$, verifying condition (ii) in Definition~\ref{D-CSC 1}.
Note that
\begin{align*}
\m(W_H\setminus W_H')&\le 1-\m\bigg(\bigcup_{h\in S_H}\sigma_hW_H'\bigg)\\
&\le 1-\m\bigg(\bigcup_{g\in S_G}\pi'_g W_H'\bigg)\\
&=1-\m\bigg(\bigcup_{g\in S_G}\pi_g W_H'\bigg)\\
&\le |S_G|(\m(V\setminus V_1)+\m(V\setminus V_2))\\
&\le |S_G|(\tilde{\tau}_G|U|^{2(100+C_G)}+ \tilde{\tau}_G|U|^{1+C_G})\le \delta_H.
\end{align*}
Put $\cV_{H, j}=\cV_{G, j}$ for all $1\le j\le n_G$, and
\[
\cV_{H, n_H}=\bigcup_{h\in T_H}((W_H\setminus W_H')\cup \sigma_h(W_H\setminus W_H')\cup \sigma_h^{-1}(W_H\setminus W_H')).
\]
Then
\[
\m (\cV_{H, n_H})\le (2|T_H|+1)\m (W_H\setminus W_H')\le 3|T_H|\delta_H,
\]
and hence
\begin{align*}
\sum_{j=1}^{n_H} \Upsilon_H(S_{H, j})\m(\cV_{H, j})&=\Upsilon_H(T_H)\m(\cV_{H, n_H})+\sum_{j=1}^{n_G} \Upsilon_H(\kappa(S_{G, j}, X))\m(\cV_{G, j})\\
&=\Upsilon_H(T_H)\m(\cV_{H, n_H})+\frac{1}{2}\sum_{j=1}^{n_G} \Upsilon_G(S_{G, j})\m(\cV_{G, j})\\
&\le 3\Upsilon_H(T_H)|T_H|\delta_H+\frac12 \le 1,
\end{align*}
verifying condition (i) in Definition~\ref{D-CSC 1}.
Let $h\in T_H$ and $w_1, w_2\in W_H$ with $\sigma_hw_1=w_2$.  If $w_1\not\in W_H'$ or $w_2\not\in W_H'$, then $(w_1, w_2)$ is an $S_{H, n_H}$-edge with both endpoints in $\cV_{H, n_H}$. Thus we may assume that $w_1, w_2\in W_H'$.
Then
$w_1\in \varphi(A)$ for some $A\in \sP_h$. Set $g=\lambda(h, A)\in T_G$. Then
\[
w_2=\sigma_hw_1=\pi'_gw_1=\pi_gw_1,
\]
and so $w_1$ and $w_2$ are connected by a path of length at most $C_G$ in which each edge is an $S_{G, j}$-edge with both
endpoints in $\cV_{G, j}$ for some $1\le j\le n_G$.
It is easily checked that such an edge is also an $S_{H, j}$-edge. This verifies condition (iii) in Definition~\ref{D-CSC 1}.
\end{proof}

\section{Topological entropy and continuous orbit equivalence} \label{S-continuous OE}

Our energies in this section will be invested in the proof of Theorem~\ref{T-continuous OE},
which in conjunction with Theorem~\ref{T-w-normal amenable}
and Proposition~\ref{P-product} yields Theorem~\ref{T-cts 1}.

\begin{theorem}\label{T-continuous OE}
Let $G\curvearrowright X$ and $H\curvearrowright Y$ be topologically free continuous actions on compact metrizable spaces,
and suppose that they are continuously orbit equivalent.
Let $\sS$ be a collection of sofic approximations for $G$,
and suppose that the action $G\curvearrowright X$ has property $\sS$-SC. Let $\Pi$ be a sofic approximation
sequence for $G$ in $\sS$. Then
\[
\hmax (H\curvearrowright Y)\ge h_\Pi (G\curvearrowright X).
\]
\end{theorem}

For the purpose of establishing the theorem we may assume, by conjugating the $H$-action by a continuous orbit
equivalence, that $Y=X$ and that the identity map on $X$ is an orbit equivalence between the two actions.
As usual we write $\kappa$ and $\lambda$, respectively,
for the cocycle maps $G\times X\rightarrow H$ and $H\times X \rightarrow G$.
As in Section~\ref{SS-invariant cts},
we take a dense $G_\delta$ subset $X_0$ of $X$ such that $X_0$ is $G*H$-invariant
and that both $G$ and $H$ act on $X_0$ freely.

Fix a compatible metric $d$ on $X$ which gives $X$ diameter no bigger than $1$.
For each $t\in G$ (resp.\ $t\in H$) we can find an $\eta_t>0$ such that
for any $x, y\in X$ with $d(x, y)\le \eta_t$ we have $\kappa(t, x)=\kappa(t, y)$ (resp.\ $\lambda(t, x)=\lambda(t, y)$),
and for a nonempty finite subset $L$ of $G$ or $H$ we set $\eta_L=\min_{t\in L}\eta_t>0$.

\begin{lemma} \label{L-construct map}
Let $L\in \overline{\cF}(H)$ and $0<\delta, \tau<1$ with $\tau<\delta^2$. Set $F=\lambda(L^2, X)\in \overline{\cF}(G)$ and
\[
\tau'=\min\big\{\eta_{L^2}\tau^{1/2}/(8|F|)^{1/2}, \tau/(22|F|^2)\big\}>0.
\]
Let $\delta_1>0$ be such that $(\tau+7\delta_1)^{1/2}\le \delta$.
Let $\pi: G\rightarrow \Sym(V)$ be an $(F, \tau')$-approximation for $G$.
Suppose that $S\in \overline{\cF}(G)$ and that $W$ is a subset of $V$ satisfying the following conditions:
\begin{enumerate}
\item $\pi_{g^{-1}}\pi_gw=w$ for all $w\in W$ and $g\in S$,

\item $\pi_g\pi_a\pi_h w=\pi_{gah} w$ for all $g, h\in S$, $a\in \lambda(L, X)$, and $w\in W$,

\item $\m ( \bigcup_{g\in S}\pi_gW ) \ge 1-\delta_1$.
\end{enumerate}
Take $0<\delta_2\le \eta_{S\cup \lambda(L, X)}$ such that for any $x, y\in X$ with $d(x, y)\le \delta_2$ one has $\max_{t\in \kappa(S, X)}d(tx, ty)\le \delta_1$.
Set
\[
\delta'=\delta_1^{1/2}\delta_2/(|S|^{1/2}|S\lambda(L, X)S|^{1/2})>0.
\]
Let $\varphi_0$ be a map in $\Map_d(F, \tau', \pi)\cap \Map_d(S\lambda(L, X)S, \delta', \pi)$ with $\varphi_0(V)\subseteq X_0$ and
$\varphi$ a map in $\Map_d(S\lambda(L, X)S, \delta', \pi)$ such that
\[
\kappa(g, \varphi_0(w))=\kappa(g, \varphi(w))
\]
for all $w\in W$ and $g\in S\lambda(L, X)S$ satisfying $\pi_gw\in W$.
Let $\sigma: H\rightarrow \Sym(V)$ be a map such that $\rho_{\Hamm}(\sigma_t, \sigma'_t)\le \tau$ for all $t\in L$, where $\sigma': H\rightarrow V^V$ is given  by
\[
\sigma'_tv=\pi_{\lambda(t, \varphi_0(v))}v
\]
for all $t\in H$ and $v\in V$.
Take $\tilde{\varphi}: V\rightarrow X$ such that $\tilde{\varphi}=\varphi$ on $W$ and such that for each $v\in \bigcup_{g\in S}\pi_gW$ one has
\[
\tilde{\varphi}(v)=\kappa(g, \varphi_0(w))\varphi(w)
\]
for some $g\in S$ and $w\in W$ with $\pi_gw=v$. Then $\tilde{\varphi}\in \Map_d(L, \delta, \sigma)$.
\end{lemma}

\begin{proof}
For each $t\in H$ set $V_t=\{v\in V: \sigma_tv=\sigma'_tv\}$. Then $\m (V_t )\ge 1-\tau$ for all $t\in L$.

Denote by $V_\varphi$ the set of all $v\in V$ satisfying $d(g\varphi(v), \varphi(\pi_gv))\le \delta_2$ for all $g\in S\lambda(L, X)S$. Then
\[
\m (V\setminus V_\varphi )\le |S\lambda(L, X)S|\bigg(\frac{\delta'}{\delta_2}\bigg)^2 = \frac{\delta_1}{|S|}.
\]
We define $V_{\varphi_0}$ in the same way, and get $\m (V\setminus V_{\varphi_0} )\le \delta_1/|S|$.
Set $W'=W\cap V_\varphi\cap V_{\varphi_0}$ and $V'=(\bigcup_{g\in S}\pi_gW)\setminus (\bigcup_{g\in S}\pi_g(V\setminus (V_\varphi\cap V_{\varphi_0})))$. Then
\begin{align*}
\m (V' )\ge \m\bigg(\bigcup_{g\in S}\pi_gW\bigg)-
|S|\cdot \m (V\setminus V_\varphi )-|S|\cdot \m (V\setminus V_{\varphi_0})\ge 1-3\delta_1.
 \end{align*}

Let $t\in L$, $v_1\in  V_t\cap V'$, and $v_2\in V'$ be such that $\sigma_tv_1=v_2$. Then we can find some $g_1,g_2\in S$ and $w_1, w_2\in W$ such that $\pi_{g_j}w_j=v_j$ and
$\tilde{\varphi}(v_j)=\kappa(g_j, \varphi_0(w_j))\varphi(w_j)$ for $j=1, 2$. Since $v_j\in V'$, we actually have $w_j\in W'$.
We also have
\begin{align*}
w_2=\pi_{g_2^{-1}}v_2
&=\pi_{g_2^{-1}}\sigma_tv_1\\
&=\pi_{g_2^{-1}}\sigma'_tv_1\\
&=\pi_{g_2^{-1}}\pi_{\lambda(t, \varphi_0(v_1))}v_1\\
&=\pi_{g_2^{-1}}\pi_{\lambda(t, \varphi_0(v_1))}\pi_{g_1}w_1\\
&=\pi_{g_2^{-1}\lambda(t, \varphi_0(v_1))g_1}w_1.
\end{align*}
Observe that
\begin{align*}
d(g_1\varphi_0(w_1), \varphi_0(v_1))=d(g_1\varphi_0(w_1), \varphi_0(\pi_{g_1}w_1))\le \delta_2\le \eta_{\lambda(L, X)}
\end{align*}
and
\begin{gather*}
d(\varphi_0(w_2), g_2^{-1}\lambda(t, \varphi_0(v_1))g_1\varphi_0(w_1))\hspace*{60mm}\\
\hspace*{10mm}\ = d(\varphi_0(\pi_{g_2^{-1}\lambda(t, \varphi_0(v_1))g_1}w_1), g_2^{-1}\lambda(t, \varphi_0(v_1))g_1\varphi_0(w_1))\le \delta_2\le \eta_S,
\end{gather*}
and hence
\[
\kappa(\lambda(t, \varphi_0(v_1)), g_1\varphi_0(w_1))=\kappa(\lambda(t, \varphi_0(v_1)), \varphi_0(v_1))=t
\]
and
\begin{align*}
\kappa(g_2^{-1}, \lambda(t, \varphi_0(v_1))g_1\varphi_0(w_1))=\kappa(g_2, g_2^{-1}\lambda(t, \varphi_0(v_1))g_1\varphi_0(w_1))^{-1}=\kappa(g_2, \varphi_0(w_2))^{-1}.
\end{align*}
Therefore
\begin{align*}
\lefteqn{\kappa(g_2^{-1}\lambda(t, \varphi_0(v_1))g_1, \varphi_0(w_1))}\hspace*{10mm}\\
&=\kappa(g_2^{-1}, \lambda(t, \varphi_0(v_1))g_1\varphi_0(w_1))\kappa(\lambda(t, \varphi_0(v_1)), g_1\varphi_0(w_1))\kappa(g_1, \varphi_0(w_1))\\
&=\kappa(g_2, \varphi_0(w_2))^{-1}t\kappa(g_1, \varphi_0(w_1)).
\end{align*}
We then get
\begin{align*}
\kappa(g_2, \varphi_0(w_2))^{-1}\tilde{\varphi}(v_2)&=\varphi(w_2)\\
&\approx_{\delta_2} g_2^{-1}\lambda(t, \varphi_0(v_1))g_1\varphi(w_1)\\
&=\kappa(g_2^{-1}\lambda(t, \varphi_0(v_1))g_1, \varphi(w_1))\varphi(w_1)\\
&=\kappa(g_2^{-1}\lambda(t, \varphi_0(v_1))g_1, \varphi_0(w_1))\varphi(w_1)\\
&=\kappa(g_2, \varphi_0(w_2))^{-1}t\kappa(g_1, \varphi_0(w_1))\varphi(w_1)\\
&=\kappa(g_2, \varphi_0(w_2))^{-1}t\tilde{\varphi}(v_1),
\end{align*}
and consequently
$d(\tilde{\varphi}(v_2), t\tilde{\varphi}(v_1))\le \delta_1$.
We conclude that
\begin{align*}
d_2(t\tilde{\varphi}, \tilde{\varphi}\sigma_t)\le (\tau+6\delta_1+\delta_1^2)^{1/2} \le (\tau+7\delta_1)^{1/2}\le \delta
\end{align*}
and hence that $\tilde{\varphi}\in \Map_d(L, \delta, \sigma)$.
\end{proof}

\begin{proof}[Proof of Theorem~\ref{T-continuous OE}]
Let $\Pi = \{ \pi_k : G\to \Sym (V_k ) \}_{k=1}^\infty$ be a sofic approximation sequence in $\sS$
with $h_\Pi(G\curvearrowright X)\ge 0$. Let $\varepsilon>0$.
To establish the theorem it is enough to show the existence of
a sofic approximation sequence $\Sigma$ for $H$ such that
$h_\Sigma(H\curvearrowright Y)\ge h_{\Pi, 2}^\varepsilon(G\curvearrowright X)-2\varepsilon$.

For each $F\in \cF(G)$, since $\kappa: G\times X\rightarrow H$ is continuous there exists a finite clopen partition ${}_F \sP$ of $X$ such that for every $g\in F$
the map $x\mapsto \kappa(g,x)$ is constant on each member of ${}_F \sP$. Define $\Upsilon: \cF(G)\rightarrow [0, \infty)$ by $\Upsilon(F)=(2/\varepsilon)\log |{}_F \sP |$.

Take a decreasing sequence $1>\delta_1>\delta_2>\dots$ converging to $0$. Take also a decreasing sequence $1>\tau_1>\tau_2>\dots>0$ with $\tau_k^2<\delta_k$ for all $k$.
Choose an increasing sequence $\{L_k\}$ in $\overline{\cF}(H)$ with union $H$.

For each $k\in \Nb$, set $F_k=\lambda(L_k^2, X)\subseteq G$ and $T_k'=\lambda(L_k, X)\subseteq F_k$.

Since $G\curvearrowright X$ has property $\sS$-SC, there is some $S\in \overline{\cF}(G)$ such that for each $k\in \Nb$, there are $C_k, n_k\in \Nb$,  $S_{k, 1}, \dots, S_{k, n_k}\in \overline{\cF}(G)$, $F^\sharp_k\in \cF(G)$, and $\delta^\sharp_k>0$ such that for any good enough sofic approximation $\pi: G\rightarrow \Sym(V)$ in $\sS$ with $\Map_d(F^\sharp_k, \delta^\sharp_k, \pi)\neq \emptyset$ there are subsets $W'$ and $\cV_j$ of $V$ for $1\le j\le n_k$
satisfying the following conditions:
\begin{enumerate}
\item $\sum_{j=1}^{n_k} \Upsilon(S_{k,j})\m(\cV_j)\le 1$,
\item $\bigcup_{g\in S}\pi_gW'=V$,
\item if $w_1, w_2\in W'$ satisfy $\pi_gw_1=w_2$ for some $g\in T_k := ST_k'S\in \overline{\cF}(G)$ then $w_1$ and $w_2$ are connected by a path of length at most $C_k$ in which each edge is an $S_{k, j}$-edge with both endpoints in $\cV_j$ for some $1\le j\le n_k$.
\end{enumerate}
Take $\varepsilon'>0$ such that for any $x, y\in X$ with $d(x, y)\le \varepsilon'$ one has
$d(gx, gy)<\varepsilon/8$ for every $g\in S$.

Fix $k\in \Nb$.
Set
\[
\tau_k'=\min\big\{\eta_{L_k^2}\tau_k^{1/2}/(8|F_k|)^{1/2}, \tau_k/(22|F_k|^2)\big\}>0.
\]
Let $0<\delta_{k, 1}<1/2$ be such that $(\tau_k+7\delta_{k, 1})^{1/2}\le \delta_k$ and $((\varepsilon/4)^2+\delta_{k, 1})^{1/2}<\varepsilon/2$.
Take $0<\delta_{k, 2}\le \eta_{T_k}$ such that for any $x, y\in X$ with $d(x, y)\le \delta_{k, 2}$ one has $\max_{t\in \kappa(S, X)}d(tx, ty)\le \delta_{k, 1}$.
Set $\delta_k'=\delta_{k, 1}^{1/2}\delta_{k, 2}/(|S|^{1/2}|T_k|^{1/2})>0$ and $\eta_k=\eta_{\bigcup_{j=1}^{n_k}S_{k, j}}$.
By Stirling's formula there is some $0<\gamma_k<\delta_{k, 1}/(3|S|)$ such that for any nonempty finite set $V$ the number of subsets of $V$ with cardinality no bigger than $\gamma_k|V|$ is at most $e^{\varepsilon |V|/2}$. Set $S_k=(\bigcup_{j=1}^{n_k}S_{k, j})^{C_k}\in \cF(G)$ and $\delta_k''=\min\{\delta_k', \min\{\eta_k, \varepsilon/16\} (\gamma_k/|S_k\cup S|)^{1/2}, \tau_k'\}>0$.

Take an $m_k\ge k$ large enough so that
\[
\frac{1}{|V_{m_k}|} \log N_{\varepsilon} (\Map_d(T_k\cup S_k\cup F_k\cup F^\sharp_k, \min\{\delta_k''/2, \delta^\sharp_k\}, \pi_{m_k} ), d_2 )
\ge \max\{0, h^{\varepsilon}_{\Pi, 2}(G\curvearrowright X)-\varepsilon\}
\]
and so that $\pi_{m_k}: G\rightarrow \Sym(V_{m_k})$ is an $(F_k, \tau_k')$-approximation for $G$ and also a good enough sofic approximation for $G$ to guarantee the existence of $W'$ and $\cV_1, \dots, \cV_{n_k}$ as above. Denote by $\tilde{V}_{m_k}$ the set of all $w\in V_{m_k}$ satisfying
\begin{enumerate}
\item[(iv)] $\pi_{m_k, e_G}w=w$,

\item[(v)] $\pi_{m_k, g^{-1}}\pi_{m_k, g}w=w$ for all  $g\in S$,

\item[(vi)] $\pi_{m_k, g}\pi_{m_k, a}\pi_{m_k, h}w=\pi_{m_k, gah}w$ for all $g, h\in S$ and $a\in T_k'$,

\item[(vii)] $\pi_{m_k, gh}w=\pi_{m_k, g}\pi_{m_k, h}w$ for all $g, h\in S_k$,

\item[(viii)] $\pi_{m_k, g}w\neq \pi_{m_k, h}w$ for all distinct $g, h$ in $T_k\cup S_k$.
\end{enumerate}
Taking $m_k$ sufficiently large, we may assume that $\m(\tilde{V}_{m_k})\ge 1-\delta_{k, 1}/(3|S|)$.

Take a $(d_2, \varepsilon)$-separated subset $\Phi$ of $\Map_d(T_k\cup S_k\cup F_k, \delta_k''/2, \pi_{m_k} )$ with maximum cardinality. Since $X_0$ is dense in $X$, we may perturb each element of $\Phi$ to obtain a $(d_2, \varepsilon/2)$-separated subset $\Phi_1$ of $\Map_d(T_k\cup S_k\cup F_k, \delta_k'', \pi_{m_k} )$ with
\[
|\Phi_1|=|\Phi|=N_{\varepsilon} (\Map_d(T_k\cup S_k\cup F_k, \delta_k''/2, \pi_{m_k} ), d_2 )
\]
such that $\varphi(V_{m_k})\subseteq X_0$ for all $\varphi\in \Phi_1$.

For each $\psi\in \Map_d(T_k\cup S_k\cup F_k, \delta_k'', \pi_{m_k} )$, using the fact that $\delta_k''\le \min\{\eta_k, \varepsilon/16\} (\gamma_k/|S_k\cup S|)^{1/2}$ we have $\m(V_{\psi})\ge 1-\gamma_k$ where
\[
V_\psi:=\{v\in V_{m_k}: d(g\psi(v), \psi(\pi_{m_k, g}v))\le \min\{\eta_k, \varepsilon/16\} \mbox{ for all } g\in S_k\cup S\}.
\]
Thus there is a subset $\Phi_2$ of $\Phi_1$ such that $V_\varphi$ is the same for all $\varphi\in \Phi_2$ and
\[
|\Phi_1|\le |\Phi_2|e^{\varepsilon |V_{m_k}|/2}.
\]
Set $W=W'\cap \tilde{V}_{m_k}\cap V_\varphi\subseteq V_{m_k}$ for $\varphi\in \Phi_2$. Then
\begin{align*}
\lefteqn{\m\bigg(\bigcup_{g\in S}\pi_{m_k, g}W\bigg)}\hspace*{10mm}\\
&\ge \m\bigg(\bigcup_{g\in S}\pi_{m_k, g}W'\bigg)-\m\bigg(\bigcup_{g\in S}\pi_{m_k, g}(V\setminus \tilde{V}_{m_k})\bigg)-\m\bigg(\bigcup_{g\in S}\pi_{m_k, g}(V\setminus V_\varphi)\bigg)\\
&\ge 1-\frac{\delta_{k, 1}}{3}-\gamma_k|S|\ge 1-\delta_{k, 1}.
\end{align*}

For each $\psi: V_{m_k}\rightarrow X$, define $\Theta(\psi)\in \prod_{j=1}^{n_k}H^{S_{k, j}\times \cV_j}$ by $\Theta(\psi)(g_j, v_j)=\kappa(g_j, \psi(v_j))$ for $1\le j\le n_k$ and $(g_j, v_j)\in S_{k, j}\times \cV_j$. Then
\[
|\Theta(X^{V_{m_k}})|\le \prod_{j=1}^{n_k}|{}_{S_{k, j}}\sP|^{|\cV_j|}=\prod_{j=1}^{n_k}e^{(\varepsilon/2)\Upsilon(S_{k, j})|\cV_j|}=e^{(\varepsilon/2)\sum_{j=1}^{n_k}\Upsilon(S_{k, j})|\cV_j|}\le e^{\varepsilon |V_{m_k}|/2}.
\]
Thus we can find a subset $\Phi_3$ of $\Phi_2$ such that $\Theta(\varphi)$ is the same for all $\varphi\in \Phi_3$ and
\[
|\Phi_2|\le |\Phi_3|e^{\varepsilon |V_{m_k}|/2}.
\]

We claim that for any  $g\in T_k$ and $w_1, w_2\in W$ with $\pi_{m_k, g}w_1=w_2$, the element $\kappa(g, \varphi(w_1))\in H$ is the same for all $\varphi\in \Phi_3$. If $w_1=w_2$, then $g=e_G$ and hence $\kappa(g, \varphi(w_1))=e_H$ for all $\varphi\in \Phi_3$. Thus we may assume that $w_1\neq w_2$.
We can find $l\le C_k$, $g_1, \dots, g_l\in G$, $w_1=w_1', w_2', \dots, w_{l+1}'=w_2$ in $V_{m_k}$ such that for each $1\le i\le l$ one has $\pi_{m_k, g_i}w_i'=w_{i+1}'$, $g_i\in S_{k, j_i}$ and $w_i', w_{i+1}'\in \cV_{j_i}$  for some $1\le j_i\le n_k$. Since $w_1'\in W\subseteq \tilde{V}_{m_k}$, we have
\[
\pi_{m_k, g_ig_{i-1}\dots g_1}w_1'=\pi_{m_k, g_i}\pi_{m_k, g_{i-1}}\dots \pi_{m_k, g_1}w_1'=w_{i+1}'
\]
for all $1\le i\le l$. In particular, $\pi_{m_k, g_lg_{l-1}\dots g_1}w_1'=w_{l+1}'=\pi_{m_k, g}w_1'$, and hence
\[
g_lg_{l-1}\dots g_1=g.
\]
Note that $w_1'\in W\subseteq V_\varphi$ for all $\varphi\in \Phi_3\subseteq \Phi_2$. For each $0\le i\le l-1$ and $\varphi\in \Phi_3$, we have $g_ig_{i-1}\dots g_1\in S_k$, and hence
$d(g_ig_{i-1}\dots g_1\varphi(w_1'), \varphi(\pi_{m_k, g_ig_{i-1}\dots g_1}w_1'))\le \eta_k$, which implies that
\[
\kappa(g_{i+1}, g_ig_{i-1}\dots g_1\varphi(w_1'))=\kappa(g_{i+1}, \varphi(\pi_{m_k, g_ig_{i-1}\dots g_1}w_1')).
\]
Then
\begin{align*}
\kappa(g, \varphi(w_1))=\kappa(g_lg_{l-1}\dots g_1, \varphi(w_1'))
&=\prod_{i=0}^{l-1}\kappa(g_{i+1}, g_ig_{i-1}\dots g_1\varphi(w_1'))\\
&=\prod_{i=0}^{l-1}\kappa(g_{i+1}, \varphi(\pi_{m_k, g_ig_{i-1}\dots g_1}w_1'))\\
&=\prod_{i=0}^{l-1}\kappa(g_{i+1}, \varphi(w_{i+1}'))\\
&=\prod_{i=0}^{l-1}\Theta(\varphi)(g_{i+1}, w_{i+1}')
\end{align*}
is the same for all $\varphi\in \Phi_3$. This proves our claim.

Fix one $\varphi_0\in \Phi_3$.  Define $\sigma_k': H\rightarrow V_{m_k}^{V_{m_k}}$   by
\[
\sigma'_{k,t}v=\pi_{\lambda(t, \varphi_0(v))}v
\]
for all $t\in H$ and $v\in V_{m_k}$. By Lemma~\ref{L-approximation for group}
there is an $(L_k, \tau_k)$-approximation $\sigma_k: H\rightarrow \Sym(V_{m_k})$ for $H$ such that
$\rho_{\Hamm}(\sigma_{k, t}, \sigma'_{k, t})\le \tau_k$ for all $t\in L_k$.
For each $\varphi\in \Phi_3$ take a $\tilde{\varphi}: V_{m_k}\rightarrow X$ such that $\tilde{\varphi}=\varphi$ on $W$ and such that for each $v\in \bigcup_{g\in S}\pi_{m_k, g}W$ one has
\[
\tilde{\varphi}(v)=\kappa(g, \varphi_0(w))\varphi(w)
\]
for some $g\in S$ and $w\in W$ with $\pi_{m_k, g}w=v$. We may require that $g$ and $w$ depend only on $v$, and not on $\varphi\in \Phi_3$. By Lemma~\ref{L-construct map} we have $\tilde{\varphi}\in \Map_d(L_k, \delta_k, \sigma_k)$.

Let $\varphi$ and $\psi$ be distinct elements in $\Phi_3$.
Since $d_2(\varphi, \psi)\ge \varepsilon/2> ((\varepsilon/4)^2+\delta_{k, 1})^{1/2}$ and $\m\big(\bigcup_{g\in S}\pi_{m_k, g}W\big)\ge 1-\delta_{k, 1}$, we have $d(\varphi(v), \psi(v))>\varepsilon/4$ for some $v\in \bigcup_{g\in S}\pi_{m_k, g}W$. Then $v=\pi_{m_k, g}w$ for some $g\in S$ and $w\in W$ such that $\tilde{\varphi}(v)=\kappa(g, \varphi_0(w))\varphi(w)$ and $\tilde{\psi}(v)=\kappa(g, \varphi_0(w))\psi(w)$. Using the fact that $w\in W\subseteq V_{\varphi}=V_{\psi}$ we have
\begin{align*}
\lefteqn{d(g\varphi(w), g\psi(w))}\hspace*{10mm}\\
&\ge d(\varphi(\pi_{m_k, g}w), \psi(\pi_{m_k, g}w))-d(\varphi(\pi_{m_k, g}w), g\varphi(w))-d(\psi(\pi_{m_k, g}w), g\psi(w))\\
&\ge \frac{\varepsilon}{4}-\frac{\varepsilon}{16}-\frac{\varepsilon}{16}=
\frac{\varepsilon}{8}.
\end{align*}
From our choice of $\varepsilon'$ we get $d(\tilde{\varphi}(w), \tilde{\psi}(w))=d(\varphi(w), \psi(w))>\varepsilon'$. Therefore $\tilde{\Phi}_3:=\{\tilde{\varphi}: \varphi\in \Phi_3\}$ is $(d_\infty, \varepsilon')$-separated and $|\tilde{\Phi}_3|=|\Phi_3|$.
Thus
\begin{align*}
\frac{1}{|V_{m_k}|}\log N_{\varepsilon'}(\Map_d(L_k, \delta_k, \sigma_k), d_\infty)\ge \frac{1}{|V_{m_k}|}\log |\tilde{\Phi}_3|
&= \frac{1}{|V_{m_k}|}\log |\Phi_3|\\
&\ge \frac{1}{|V_{m_k}|}\log |\Phi_1|-\varepsilon\\
&\ge h^{\varepsilon}_{\Pi, 2}(G\curvearrowright X)-2\varepsilon.
\end{align*}

Now $\Sigma=\{\sigma_k\}_{k\in \Nb}$ is a sofic approximation sequence for $H$.
For any finite set $L\subseteq H$ and $\delta>0$, we have $L\subseteq L_k$  and $\delta>\delta_k$ for all large enough $k$, and hence
\begin{align*}
\varlimsup_{k\to \infty}\frac{1}{|V_{m_k}|}\log N_{\varepsilon'}(\Map_d(L, \delta, \sigma_k), d_\infty)&\ge \varlimsup_{k\to \infty}\frac{1}{|V_{m_k}|}\log N_{\varepsilon'}(\Map_d(L_k, \delta_k, \sigma_k), d_\infty)\\
&\ge h^{\varepsilon}_{\Pi, 2}(G\curvearrowright X)-2\varepsilon.
\end{align*}
Taking infima over $L$ and $\delta$, we obtain
\[
h_\Sigma(H\curvearrowright X)\ge h_{\Sigma, \infty}^{\varepsilon'}(H\curvearrowright X)\ge h^{\varepsilon}_{\Pi, 2}(G\curvearrowright X)-2\varepsilon.
\]
\end{proof}

\section{Measure entropy and bounded orbit equivalence} \label{S-bounded OE}

In this final section we establish Theorem~\ref{T-bounded OE},
which in conjunction with Proposition~\ref{P-product} yields Theorem~\ref{T-measure 1}.

For a general reference on the C$^*$-algebra theory and terminology used
in the following proof, see \cite{Mur90}.

\begin{lemma} \label{L-model}
Let $G\curvearrowright (X, \mu)$ and $H\curvearrowright (Y, \nu)$ be orbit equivalent
free p.m.p.\ actions and
suppose that $H\curvearrowright (Y, \nu)$ is uniquely ergodic. Then there are a zero-dimensional compact metrizable space $Z$, a continuous action $G*H \curvearrowright Z$, and a $G*H$-invariant Borel probability measure $\mu_Z$ on $Z$ of full support
such that
\begin{enumerate}
\item $G\curvearrowright (Z, \mu_Z)$
is measure conjugate to $G\curvearrowright (X, \mu)$ and $H\curvearrowright (Z, \mu_Z)$
is measure conjugate to $H\curvearrowright (Y, \nu)$,

\item $H\curvearrowright Z$ is uniquely ergodic,

\item there is a $G*H$-invariant Borel subset $Z_0$ of $Z$ with $\mu_Z(Z_0)=1$ such that $Gz=Hz$ for every $z\in Z_0$.
\end{enumerate}
If furthermore $G\curvearrowright (X, \mu)$ and $H\curvearrowright (Y, \nu)$ are boundedly orbit equivalent then we may demand that both $G\curvearrowright Z_0$ and $H\curvearrowright Z_0$ be free and that the cocycles $\kappa: G\times Z_0\rightarrow H$ and $\lambda: H\times Z_0\rightarrow G$ extend to continuous maps $G\times Z\rightarrow H$ and $H\times Z\rightarrow G$, so that
$G\curvearrowright Z$ and $H\curvearrowright Z$ are continuously orbit equivalent.
\end{lemma}

\begin{proof}
We may assume that $(X, \mu)=(Y, \nu)$ and that $Gx=Hx$ for every $x\in X$.
Denote by $\sB$ the $\sigma$-algebra of Borel subsets of $X$.
Denote by $\varphi$ the mean $f\mapsto \int_X f\, d\mu$ on $L^\infty(X, \mu)$.

Let $V$ be a finite subset of $\sB$. We claim that for every $\varepsilon>0$ there is a finite subset $W$ of $\sB$ containing $V$ such that for any mean $\psi$ on $L^\infty(X, \mu)$ satisfying $\psi(s1_A)=\psi(1_A)$ for all $s\in H$ and
$A\in W$
one has $|\varphi(1_A)-\psi(1_A)|<\varepsilon$ for all $A\in V$. Suppose to the contrary
that for some $\eps > 0$ and every finite subset $W$ of $\sB$ containing $V$
there is a mean $\psi_W$ on $L^\infty(X, \mu)$ satisfying $\psi_W(s1_A)=\psi_W(1_A)$ for all $s\in H$ and $A\in W$
and $\max_{A\in V}|\varphi(1_A)-\psi_W(1_A)|\ge \varepsilon$. Then any cluster point $\psi$ of the net
$\{ \psi_W \}$ (with index directed by inclusion) is $H$-invariant and $\max_{A\in V}|\varphi(1_A)-\psi(1_A)|\ge \varepsilon$. This contradicts the unique ergodicity of $H\curvearrowright (X, \mu)$, thus verifying our claim.

For any countable subset $V$ of $\sB$, writing $V$ as the union of an increasing sequence $\{V_k\}_k$ of finite subsets of $V$ and taking a  sequence $\{\varepsilon_k\}_k$ of positive numbers tending to $0$, we conclude from above that there is a countable subset $W$ of $\sB$ containing $V$ such that for any mean $\psi$ on $L^\infty(X, \mu)$ satisfying $\psi(s1_A)=\psi(1_A)$ for all $s\in H$ and
$A\in W$
one has $\varphi(1_A)=\psi(1_A)$ for all $A\in V$.

For a given countable set $\sA\subseteq \sB$, denote by $\sA'$ the $G*H$-invariant subalgebra of $\sB$ generated by $\sA$,
which is again countable.
Take a countable subset $\sA_1$ of $\sB$ such that for any distinct $x, y\in X$ one has $1_A(x)\neq 1_A(y)$ for some $A\in \sA_1$. Inductively, having constructed a countable subset $\sA_k$ of $\sB$, we take a countable subset $\sA_{k+1}$ of $\sB$ containing $\sA'_k$ such that for any mean $\psi$ on $L^\infty(X, \mu)$ satisfying $\psi(s1_A)=\psi(1_A)$ for all $s\in H$ and
$A\in \sA_{k+1}$,
one has $\varphi(1_A)=\psi(1_A)$ for all $A\in \sA'_k$.

Now we put $\sA=\bigcup_k\sA_k$. This is a countable $G*H$-invariant subalgebra of $\sB$. For any mean $\psi$ on $L^\infty(X, \mu)$ satisfying $\psi(s1_A)=\psi(1_A)$ for all $s\in H$ and
$A\in \sA$,
one has $\varphi(1_A)=\psi(1_A)$ for all $A\in \sA$.
Denote by $\fA$ the $G*H$-invariant unital C$^*$-subalgebra of $L^\infty(X, \mu)$ generated by
the functions $1_A$ for $A\in \sA$. Then $\fA$ is the closure of the linear span of the functions $1_A$ for $A\in \sA$ in $L^\infty(X, \mu)$. Thus every state of $\fA$ is determined by its values on the functions $1_A$ for $A\in \sA$. Since every state of $\fA$ extends to a mean of $L^\infty(X, \mu)$, we conclude that $\varphi|_\fA$ is the unique $H$-invariant state on $\fA$.

Define a $G*H$-action on $\{0, 1\}^{\sA\times (G*H)}$ by $(sw)_{A,t} = w_{A,s^{-1} t}$
for $w\in \{0, 1\}^{\sA\times (G*H)}$, $A\in\sA$, and $s,t\in G*H$,
and consider the $G*H$-equivariant Borel map $\pi: X\rightarrow \{0, 1\}^{\sA\times (G*H)}$
given by $\pi(x)_{A, t}=1_A(t^{-1}x)=1_{tA}(x)$ for $x\in X$, $A\in \sA$, and $t\in G*H$.
Since $\sA_1\subseteq \sA$, the map $\pi$ is injective and hence is a Borel isomorphism from $X$ to $\pi(X)$ \cite[Corollary~15.2]{Kec95}.
Put $\mu_Z=\pi_*\mu$ and $Z=\supp(\mu_Z)$. Then $Z$ is zero-dimensional and $\mu_Z$ is a $G*H$-invariant Borel probability measure on $Z$ of full support. Put $Z_0=\pi(X)\cap Z$. Then $Z_0$ is $G*H$-invariant with $\mu_Z(Z_0)=1$, and $Gz=Hz$ for all $z\in Z_0$. The pull-back map $\pi^*: C(Z)\rightarrow L^\infty(X, \mu)$ is a $G*H$-equivariant $*$-homomorphism. From the Stone--Weierstrass theorem we get $\pi^*(C(Z))=\fA$. Since $Z$ is the support of $\pi_*\mu$, the map $\pi^*$ is injective and hence is an isomorphism from $C(Z)$ to $\fA$. Thus $C(Z)$ has a unique $H$-invariant state, which means that $H\curvearrowright Z$ is uniquely ergodic.

Now assume that $G\curvearrowright (X, \mu)$ and $H\curvearrowright (Y, \nu)$ are boundedly orbit equivalent.
By passing to suitable invariant subsets we may assume
that $G\curvearrowright X$ and $H\curvearrowright X$ are both genuinely free and that the cocycles $\kappa': G\times X\rightarrow H$ and $\lambda': H\times X\rightarrow G$ are both bounded. Adding more sets to $\sA_1$, we may assume that for every $t\in G$
(resp.\ $t\in H$) there is a finite partition $\sP$ of $X$ contained in $\sA_1$ such that $\kappa'$ (resp.\ $\lambda'$) is constant on $\{t\}\times P$ for every $P\in \sP$. Then we can extend $\kappa$ (resp.\ $\lambda$) continuously to $G\times Z\rightarrow H$ (resp.\ $H\times Z\rightarrow G$).
\end{proof}

\begin{theorem}\label{T-bounded OE}
Let $G\curvearrowright (X,\mu)$ and $H\curvearrowright (Y,\nu)$ be free p.m.p.\ actions
which are boundedly orbit equivalent.
Let $\sS$ be a collection of sofic approximations for $G$.
Suppose that $G$ has property $\sS$-SC
and that the action $H\curvearrowright (Y,\nu)$ is uniquely ergodic. Let $\Pi$ be a sofic approximation sequence in $\sS$. Then
\[
\hmax_\nu (H\curvearrowright Y)\ge h_{\Pi,\mu} (G\curvearrowright X).
\]
\end{theorem}

\begin{proof}
Combine Lemma~\ref{L-model}, Theorem~\ref{T-continuous OE}, and the
variational principle (Theorem~10.35 in \cite{KerLi16}).
\end{proof}

\end{document}